\documentclass[10pt]{amsart} 
\usepackage[english]{babel}
\usepackage{amsmath}
\usepackage{amssymb}
\usepackage{amsthm}
\usepackage{amscd}
\usepackage{xcolor}
\usepackage{fancyhdr}
\usepackage{hyperref}
\usepackage{cleveref}
\usepackage[bottom=3cm,top=3.5cm, left=3cm, right=3cm]{geometry}
\usepackage[normalem]{ulem}
\usepackage{amsfonts}
\usepackage{amsbsy}
\usepackage[utf8]{inputenc}
\usepackage[shortlabels]{enumitem}
\usepackage{cancel}
\usepackage{xspace}
\usepackage{url}
\usepackage{wrapfig}
\usepackage{enumitem}
\usepackage{multicol}
\usepackage{moreenum}
\usepackage{xcolor}
\usepackage{subcaption}
\usepackage{tikz}
\usepackage{array,multirow}
\usepackage[all,cmtip]{xy}
\definecolor{LinkColor}{rgb}{0,0,0} 

\newtheorem*{Main Theorem} {Main Theorem}
\newtheorem*{Theorem A} {Theorem A}
\newtheorem*{Theorem B} {Theorem B}
\newtheorem*{Theorem C} {Theorem C}
\newtheorem*{Conjecture A}{Conjecture A}
\newtheorem*{Conjecture B}{Conjecture B}

\newtheorem*{Question B}   {Question B}

\newtheorem {theorem}    {Theorem}[section]
\newtheorem {lemma}      [theorem]    {Lemma}

\newtheorem {proposition}[theorem]    {Proposition}

\newtheorem {question}   {Question}

\newtheorem {fact}       {Fact}

\newtheorem{maintheorem}{Theorem}

\theoremstyle{definition}

\numberwithin{equation}{section}

\newcounter{DM@bibnum}

\DeclareMathOperator{\GL}{GL}

\DeclareMathOperator{\Aut}{Aut}

\DeclareMathOperator{\Irr}{Irr}

\newcommand{\F}{\mathbb{F}}
\newcommand{\Z}{\textup{Z}}

\newcommand{\C}{\textup{C}}
\newcommand{\N}{\textup{N}}
\newcommand{\NN}{\mathbb{N}}

\newcommand{\MM}{\mathbb{M}}
\newcommand{\GK}{\Gamma_{\textup{GK}}}

\newcommand{\Core}{\operatorname{Core}}
\newcommand{\GEN}[1]{\left\langle #1 \right\rangle}

\newcommand{\qand}{\quad \text{and} \quad}
\newcommand{\cut}{\textsf{cut}\xspace}

\DeclareMathOperator{\Ind}{{\rm Ind}}      

\title{Gruenberg-Kegel graphs with four edges \\ of finite solvable cut groups  }

\author{Sara C. Debón}
\address{Departamento de Matem\'aticas, Universidad de Murcia, Spain}
\email{sara.cebelland@um.es}

\author{Diego García-Lucas}
\address{Departamento de Matemáticas, Universidade de Santiago de Compostela, Spain}
\email{diego.garcia.lucas@usc.es}

\author{\'{A}ngel del R\'{\i}o} 
\address{Departamento de Matem\'aticas, Universidad de Murcia, Spain}
\email{adelrio@um.es}

\thanks{Partially supported by Grant PID2024-155576NB-I00 funded by MICIU/AEI/ 10.13039/501100011033 /FEDER, UE, and by Grant 22004/PI/22 funded by Fundación Séneca of Región de Murcia, Spain.}

\begin{document}
	\maketitle

\begin{abstract}
We prove that there is a unique graph with four edges which is the Gruenberg-Kegel graph of a solvable \cut group. This contributes to the classification of the Gruenberg-Kegel graphs of solvable \cut groups initiated in \cite{BKMdR}.
\end{abstract}

The \emph{Gruenberg-Kegel graph} (also known as prime graph) of a group $G$ is the one having as vertices the primes ocurring as the order of some element of $G$ and edges $p-q$ whenever $G$ has an element of order $pq$. The Gruenberg-Kegel graph of $G$ is denoted $\GK(G)$ and in the remainder of the paper we call it \emph{GK-graph}, for brevity. This graph was introduced by Karl Gruenberg and Otto Kegel who proved that if $G$ is solvable then $\GK(G)$ is disconnected if and only if $G$ is Frobenius or 2-Frobenius (see \cite{Williams}, which quotes an unpublished manuscript of Gruenberg and Kegel). Since the publication of this paper, the GK-graph has received a lot of attention and in particular, the properties of the GK-graph of a group in a class $\mathcal{C}$ of groups,  or the realizability of a graph as the GK-graph of an object of $\mathcal{C}$, have been object of study  \cite{AKK10,BurnessCovato15,CameronMaslova22,GKKM14,GrechkoseevaVasilev22,GKLNS,MP16,ZM13}. 

A group is said to be \cut if the \textbf{c}entral \textbf{u}nits of its integral group ring are \textbf{t}rivial. 
It is well known that a finite group $G$ is \cut if and only if it is inverse semi-rational, i.e. if and only if for every $g\in G$, all the generators of the group generated by $g$ are conjugated to $g$ or $g^{-1}$ (see \cite{Ritter1990} or \cite[Corollary~7.1.15]{JespersdelRioGRG1}). 
In particular every rational group is \cut. 
The study of rational groups is a classical topic in representation theory \cite{FeitSeitz1989,Gow,Hegedus2005,Kletzing1984,Thompson2008}. 
Besides the connections of \cut groups with units of integral group rings, they have a role in representation theory because they can also be characterized as the groups for which the field of character values of every irreducible character is an imaginary quadratic extension of the rationals. So the interest for \cut groups and other generalizations of rational groups have received some attention during the last decades  \cite{Bachle2018,BCJM,BakshiMaheshwaryPassi17,ChillagDolfi,Mah18,Mor22,Navarro09,delRioVergani,Tent12,Tre19}.

The classification of the GK-graphs of solvable rational groups and of solvable \cut groups was almost achieved in \cite{BKMdR}. 
  For rational solvable groups the classification was completed in \cite{DebonGLucasdelRio24,DebonGLucasdelRio25}.

The aim of this paper is to contribute to the classification of the GK-graphs of solvable \cut groups. 
In \cite{BKMdR} the graphs with at most three vertices which occur as the GK-graph of a solvable \cut group were classified. It was also proved that a graph with more than three vertices realizable as the GK-graph of a solvable \cut group is one of a list of seven graphs, of which three are indeed the GK-graph of a solvable \cut group. Of these three, one has four edges, another has five and the last one has six edges. However, the question was not settled for the remaining four graphs and the authors posed this in the form of the following question:

\begin{question}\label{QuestionE}
	\cite[Question E]{BKMdR} Which of the following graphs are realizable as the GK-graph of a solvable \cut group?
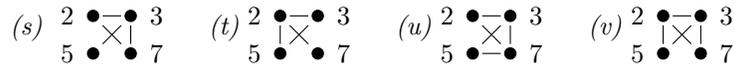
	\begin{center}
		\begin{tabular}{cccc}
			(s) \begin{tikzpicture}[local bounding box=bb,baseline=(bb.center)]
				\node[label=west:{$2$}] at (0,0.5) (2){};
				\node[label=east:{$3$}] at (0.5,0.5) (3){};
				\node[label=west:{$5$}] at (0,0) (5){};
				\node[label=east:{$7$}] at (0.5,0) (7){};
				\foreach \p in {2,3,5,7}{
					\draw[fill=black] (\p) circle (0.075cm);
				}
				\draw (2)--(3);
				\draw (2)--(7);
				\draw (3)--(5);
				\draw (3)--(7);
			\end{tikzpicture} \quad
			(t)\begin{tikzpicture}[local bounding box=bb,baseline=(bb.center)]
				\node[label=west:{$2$}] at (0,0.5) (2){};
				\node[label=east:{$3$}] at (0.5,0.5) (3){};
				\node[label=west:{$5$}] at (0,0) (5){};
				\node[label=east:{$7$}] at (0.5,0) (7){};
				\foreach \p in {2,3,5,7}{
					\draw[fill=black] (\p) circle (0.075cm);
				}
				\draw (2)--(3);
				\draw (2)--(5);
				\draw (2)--(7);
				\draw (3)--(5);
			\end{tikzpicture} \quad
			(u)\begin{tikzpicture}[local bounding box=bb,baseline=(bb.center)]
				\node[label=west:{$2$}] at (0,0.5) (2){};
				\node[label=east:{$3$}] at (0.5,0.5) (3){};
				\node[label=west:{$5$}] at (0,0) (5){};
				\node[label=east:{$7$}] at (0.5,0) (7){};
				\foreach \p in {2,3,5,7}{
					\draw[fill=black] (\p) circle (0.075cm);
				}
				\draw (2)--(3);
				\draw (2)--(7);
				\draw (3)--(5);
				\draw (3)--(7);
				\draw (5)--(7);
			\end{tikzpicture} \quad
			(v)\begin{tikzpicture}[local bounding box=bb,baseline=(bb.center)]
				\node[label=west:{$2$}] at (0,0.5) (2){};
				\node[label=east:{$3$}] at (0.5,0.5) (3){};
				\node[label=west:{$5$}] at (0,0) (5){};
				\node[label=east:{$7$}] at (0.5,0) (7){};
				\foreach \p in {2,3,5,7}{
					\draw[fill=black] (\p) circle (0.075cm);
				}
				\draw (2)--(3);
				\draw (2)--(5);
				\draw (2)--(7);
				\draw (3)--(5);
				\draw (3)--(7);
			\end{tikzpicture} 
		\end{tabular} 
	\end{center}
\end{question}

In this paper we answer the question for the first two, in the negative. 
Combining this with the main theorem of \cite{BKMdR} we obtain the following theorem:	

\begin{maintheorem}\label{Main}
	The only graph with four edges which is the GK-graph of a solvable \cut group is 
	
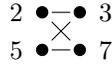
\begin{center}
		\begin{tikzpicture}[local bounding box=bb,baseline=(bb.center)]
			\node[label=west:{$2$}] at (0,0.5) (2){};
			\node[label=east:{$3$}] at (0.5,0.5) (3){};
			\node[label=west:{$5$}] at (0,0) (5){};
			\node[label=east:{$7$}] at (0.5,0) (7){};
			\foreach \p in {2,3,5,7}{
				\draw[fill=black] (\p) circle (0.075cm);
			}
			\draw (2)--(3);
			\draw (2)--(7);
			\draw (3)--(5);
			\draw (5)--(7);
		\end{tikzpicture}  
	\end{center}
\end{maintheorem}

Observe that the two hypothesis (solvable and \cut) of the theorem are needed. 
For example, by the main result of \cite{GKLNS} every graph with four vertices and four edges is the GK-graph of a finite solvable group. 
On the other hand, $S_9$ is rational and therefore \cut, and its GK-graph is (t). 

%

The graph in \Cref{Main} is the GK-graph of the direct product of the Frobenius groups $C_5^2\rtimes Q_8$ and $C_7\rtimes C_3$. This group is clearly solvable and it is \cut because $C_5^2\rtimes Q_8$ is rational and $C_7\rtimes C_3$ is \cut (see \cite[Lemma~2.6(5)]{BKMdR}). Moreover, by the main theorem of \cite{BKMdR}, the graphs (s) and (t) are the only other possible graphs with four edges which can occur as the GK-graph of a solvable \cut group. So to prove \Cref{Main} we just have to show that they are not the GK-graph of a solvable \cut group.

The paper is organized as follows: In \Cref{SectionNotation} we introduce the notation and terminology used in the paper. \Cref{SectionTechnical} and \Cref{SectionTechnicalCut} are dedicated to technical lemmas, first about not-necessarily \cut groups, and then about \cut groups. In \Cref{Section-s} we prove that graph (s) is not the GK-graph of any solvable \cut group, and in \Cref{Section-t} we obtain the same result for graph (t).

	\section{Notation}\label{SectionNotation}

By default all groups are finite and all modules are right modules.
	
We begin with some common notation and terminology for a group $G$: 
$\pi(G)$ denotes the set of primes occurring as the order of an element of $G$; the center, the Fitting subgroup, and the group of automorphisms of $G$ are denoted $Z(G)$, $F(G)$ and $\Aut(G)$, respectively; the Fitting length of $G$ is denoted $l_F(G)$;  
\[1=F_0(G)\subseteq F_1(G)=F(G)\subseteq F_2(G), \dots\]
is the Fitting series of $G$ and 
\[1=Z_0(G)\subseteq Z_1(G)=Z(G)\subseteq Z_2(G), \dots \]
and 
\[G=\gamma_1(G)\supseteq \gamma_2(G)=G'\supseteq \gamma_3(G), \dots \]
are the upper and lower central series of $G$, respectively.

If $H$ is a subgroup of $G$ then $\N_G(H)$ denotes the normalizer of $H$ in $G$. 
For each $g\in G$, $|g|$ denotes the order of $g$, and if $h$ is another element of $G$ then $g^h=h^{-1}gh$ and $(g,h)=g^{-1}h^{-1}gh$. 
If $p$ is a prime integer then $g_p$ denotes the $p$-part of $g$, and more generally, if $\pi$ is a set of prime integers, then $g_{\pi}$ denotes the $\pi$-part of $G$. 
Moreover, $G_p$ denotes an arbitrary Sylow subgroup of $G$, $G_{\pi}$ denotes an arbitrary  $\pi$-Hall subgroup of $G$ and $O_p(G)=F(G)_p$. As all the groups considered in this paper are finite and solvable, $G_{\pi}$ always exists and is unique up to $G$-conjugation. Finally, if $G$ is a $p$-group, then 
	\[\Omega(G)=\GEN{g\in G: g^p=1}.\]

 If $G$ acts on a set $X$ and $A$ is a subset of $X$, then we denote 
	\[\C_G(A)=\{g\in G : ag=a, \text{ for every } a\in A\}.\]
When this notation is used with $A$ a subset of group $H$ containing $G$, then we are considering $G$ acting on  $H$ by conjugation and hence $\C_G(A)$ is the centralizer of $A$ in $G$. More generally, if $N$ is a normal subgroup of $H$ and $A$ is a subset of $H/N$, then conjugation  by elements  of $G$  on $H$ induces an action of $G$ on $H/N$, and this is the action considered when the notation $\C_G(A)$ is used in this context. If $F$ is a field, then $FG$ denotes the group algebra of $G$ with coefficients in $F$ and if $V$ is an $FG$-module, then the  action of $G$ on $V$ allows using the notation $\C_G(A)$ for $A\subseteq V$. 
		
If $g\in G$, then we say that $g$ is 
	
\quad \emph{rational} in $G$ if every generator of $\GEN{g}$ is conjugate to $g$ in $G$;
	
\quad \emph{inverse semi-rational} in $G$ if every generator of $\GEN{g}$ is conjugate to $g$ or $g^{-1}$ in $G$;
	
\quad \emph{real} in $G$ if it is conjugate to $g^{-1}$ in $G$.
	
Finally	$G$ is said to be 

\quad \emph{rational} if every element of $G$ is rational in $G$, 

\quad \emph{inverse semi-rational} or \cut if every element of $G$ is inverse semi-rational in $G$. The name \cut comes from the fact that a finite group $G$ is \cut if and only if all the central units of the integral group ring of $G$ are trivial.

	Clearly $g$ is rational in $G$ if and only if $[\N_G(\GEN{g}):\C_G(g)]=|\Aut(\GEN{g})|$.
	On the other hand, $g$ is inverse semi-rational in $G$ if and only if it is rational or $[\N_G(\GEN{g}):\C_G(g)]=|\Aut(\GEN{g})|/2$ and $g$ is not real in $G$.
	So $g$ is rational in $G$ if and only if it is inverse semi-rational and real in $G$. 

As usual $G=N\rtimes H$, represents a semidirect product of $N$ and $H$ (with $N$ normal in $G$). Moreover, $G=N\rtimes_{\text{Fr}} H$, represents a Frobenius group with Frobenius kernel $N$ and Frobenius complement $H$. 
	
For a positive integer $n$, $C_n$ denotes a cyclic group of order $n$, $D_n$ a dihedral group of order $n$, $Q_n$ a quaternion group of order $n$ and $\Sigma_n$, the symmetric group on $n$ symbols. Finally, the following Frobenius group, which is the unique rational group of order 200, will play a relevant role in this paper:
	\[\MM=C_5^2\rtimes_{\text{Fr}} Q_8=\text{SG}[200,44].\]
The notation $\text{SG}[n,m]$ represents the $m$-th group of order $n$ in the {\sf GAP} small groups library \cite{GAP4}.

 Let $F$ be a field, let $G$ be a group and let $V$ be a right $FG$-module. 
If $g\in G$ and $\alpha\in F$, then let us denote
	\[V_g(\alpha)=\{v\in V : vg=\alpha v\}.\]
We say that $V_{FG}$ has the \emph{eigenvector property} if for every $v\in V$ and every $\alpha\in F\setminus \{0\}$ there is $g\in G$ such that $vg=\alpha v$, i.e. $v\in V_g(\alpha)$.
In other words, $V_{FG}$ satisfies the eigenvector property if for every $\alpha\in F$, $V=\cup_{g\in G} V_g(\alpha)$. 

Let $H$ be a subgroup of $G$. Then $V_H$ denotes the $FH$-module obtained by restriction of scalars. Let now $W$ be an $FH$-module. Then $W^G$ denotes the induced $FG$-module.
In some cases the latter is denoted $\Ind^G_H(W)$, to avoid ambiguity with the subgroup $H$.  
For $g\in G$, let $W^g$ denote the $F H$-module which is   equal to  $W$ as $F$-vector space and  multiplication by $h\in H$ on $W^g$ is as multiplication by $ghg^{-1}$ on $W$.
Observe that if $T$ is a right transversal of $H$ in $G$, then $W^G=\oplus_{g\in T} W^g$.
If $H$ is normal in $G$, then $I_G(W)$ denotes the inertia group of $W$ in $G$, i.e. 
\[I_G(W)=\{g\in G : W^g\cong W\}.\]

	We identify elementary abelian $p$-groups with vector spaces over the field $\F_p$ with $p$ elements. Similarly, we identify $\F_p G$-modules with elementary abelian $p$-groups $V$ with an action of $G$ by linear maps.
	In that case, $V\rtimes G$ denotes the corresponding semidirect product. 
	In multiplicative group theoretical notation, $V_g(\alpha)=\{v\in V : v^g=v^\alpha\}$ and the eigenvector property of $V_{\F_p G}$ takes the following form:  For every $k\in \{1,\dots,p-1\}$ and every $v\in V$, there is $g\in G$ such that $v^g=v^k$. 
	In that case, we abbreviate the terminology by saying that $V_G$ has the eigenvector property.
	Observe  that  $V_G$ has the eigenvector property if and only if every element of $V$ is rational in $V\rtimes G$.
	Moreover, if $r$ represents a generator of the group of units of $\F_p$, then $V_G$ has the eigenvector property if and only if for every $v\in V$ there is $g\in G$ such that $v^g=v^r$.

\section{Technical results}\label{SectionTechnical}

In this section we present a list of technical results which will be used in the subsequent sections. 
We start with a known result about the Fitting subgroup of a solvable group on which every normal abelian subgroup is cyclic.

\begin{proposition}\cite[Corollary~1.10]{ManzWolf}\label{StructureFED}
	Let $G$ be a solvable group such that every normal abelian subgroup of $G$ is cyclic. Let $F=F(G)$, let $Z$ be the socle of $Z(F(G))$ and let $A=\C_G(Z)$. 
	Then there exist normal subgroups $E$ and $D$ of $G$ with the following properties:
	
	\begin{enumerate}
		\item\label{FZD} $F=ED=\C_A(E/Z)$, $Z=E\cap D$ and $D=\C_F(E)=\C_G(E)$;
		\item for each prime divisor $q$ of $|E|$, a Sylow $q$-subgroup of $E$ is cyclic of order $q$ or extra-special of exponent $q$ or $4$;
		\item \label{Ucyclic} $D$ has a cyclic subgroup $U$ such that $[D:U]\le 2$, $U\trianglelefteq G$ and $\C_D(U)=U$;
	\end{enumerate}
\end{proposition}

\begin{lemma}\label{FrobeniusFaithful}
		Let $F$ be a field, let $G$ be a Frobenius group with Frobenius complement $H$ and assume that the order of $G$ is not divisible by the characteristic of $F$.
		If $M$ is a faithful $FG$-module, then for every $x\in H$ there exists $m\in M\setminus \{0\}$ with $mx=m$. In particular, $H$ does not act fixed-point-freely on $M$.
	\end{lemma}
	\begin{proof}
		Let $N$ be the Frobenius kernel of $G$. As $M$ is faithful, it has a simple $FG$-submodule $V$ on which the action of $N$ is non trivial. 
		Let $\mathfrak{X}$ be the $F$-representation of $G$ associated to $V$ and let $E$ be an splitting field of $G$.  Then $V\otimes_F E$ is the direct sum of the Galois conjugates of some absolutely irreducible $EG$-module $W$ and its representation $\mathfrak{X}^E$ is simply $\mathfrak{X}$, considered as $E$-representation. 
		Hence, $N$ is not in the kernel of $W$. 
		By \cite[Theorem~18.7]{Huppert1998}, there exists some $EN$-module $U$ such that $W=\Ind_N^G(U)= \bigoplus_{h\in H}U\otimes h$. 
		Let $u\in U\setminus \{0\}$ and $x\in H$. Then $w=u\otimes \sum_{h\in H} h\ne 0$ and $wx=w$. Therefore, the action of $x$ on $W$ is not fixed-point-free, so $1$ is an eigenvalue of the representation of $W$ in $x$. Using the decomposition of $V\otimes_ F E$, we conclude that $1$ is also an eigenvalue of $\mathfrak{X}^E(x)$. As $\mathfrak{X}^E(x)=\mathfrak{X}(x)$ we conclude that there exists some $m\in V\setminus \{0\}$ with $mx=m$.
\end{proof}

\begin{lemma}\label{FabelianSylow}
	Let $G$ be a finite group such that $\pi(G)=\{2,p\}$ with $p$ odd and $G_2\cong Q_8$. 
	Suppose that $G$ has not elements of order $4p$ and every $p$-element of $G$  is real in $G$.
	Then $F(G)$ is an abelian Sylow $p$-subgroup of $G$ and $a^x=a^{-1}$ for every element $x\in G$ of order $2$ and every $a\in F(G)$.
\end{lemma}

\begin{proof}  
	Let $F=F(G)$. As the $p$-elements of $G$ are real, $F$ cannot contain a Sylow $2$-subgroup of $G$. Hence $F_2$ is a cyclic $2$-group and so it does not have  non-trivial  automorphisms of odd order. 
	   As $G$ is solvable, $\C_G(F)\subseteq F$ and therefore, $F_p\neq 1$ and the order of $F(G/F_p)$ is not multiple of $p$.
	Suppose that   $F_p$  is not a Sylow $p$-subgroup of $G$. Then    $G/F_p$ satisfies the hypothesis of the lemma. So again the $p$-part of $F(G/F_p)$ is non-trivial, a contradiction.
	This proves that $F_p$ is a Sylow $p$-subgroup of $G$. 
	
	Let $a\in F_p$. By the assumptions $G$ has some $2$-element $x$ such that $a^{x}=a^{-1}$. 
	We claim that if $a\ne 1$, then $|x|=2$. Otherwise $|x|=4$, so there is another $2$-element $y\in G$ such that $\GEN{x,y}$ is a Sylow $2$-subgroup of $G$. Let $A=\GEN{a,a^y}$ and $B=\GEN{A,x,y}$. Then $A$ is normal in $B$ and hence so is $A'$. Note that $a\not \in A'$ as $A\neq A'$. We will prove that $ayA', axyA' $ or $a^ya yA'$ has order divisible by $4p$, which contradicts one of the assumptions.
	Indeed, if $a^yA'\in \GEN{aA'}$, then either $aA'$ commutes with $yA'$, or $a^yA'=a^{-1}A'$, because $a$ is a $p$-element and $x^2=y^2$ commutes with $a$. In the latter case $aA'$ commutes with $xyA'$, so either $ayA'$ or $axyA'$ have order divisible by $4p$. 
	Otherwise $\GEN{aA', a^yA'}$ is an abelian $p$-group of rank $2$, and $aa^yA'$ is non-trivial and commutes  with $yA'$, so the order of $aa^yyA'$ is multiple of $4p$. This finishes the proof of the claim.

	We conclude that the order of $F$ is odd, since otherwise there would be a unique (and hence central)  element of order $2$ in $G$, contradicting the claim. 
	
	Let $x$ be an element of order $4$ in $G$. 
We claim that the map $f:F\to F$, given by $f(a)=a^xa$ is a bijection and $a^{x^2}=a^{-1}$ for every $a\in F$. Indeed, for the first part of the claim let $a,b\in F$ with $f(a)=f(b)$ and let $c=b^{-1}a$. Then 
	\[b^xb=a^xa=(bc)^xbc= b^xc^xbc=b^xbc^{xb}c.\]
	Thus $c^{xb}=c^{-1}$, but $|xb|=4$, as $|x|=4$ and $G$ has not elements of order $4p$. Then $c=1$, by the previous claim, so $a=b$.  
	For the second part of the claim, let $a\in F$. Since $f$ is a bijection, $a=b^xb$ for some $b\in F$. Since $G$ has no elements of order $4p$, the element $xb$ has order $4$, so $(xb)^2= x^2b^x b = x^2a$ has order $2$ and therefore $a^{x^2}=a^{-1}$. This finishes the proof of the claim. 
	Finally, considering the automorphism of $F$ given by conjugation by $x^2$, we deduce
	that the map $a \longmapsto  a^{-1}$ is an automorphism of $F$ and hence $F$ is abelian.
\end{proof}

\begin{lemma}\label{HAbelCic} Let $G$ be a finite group and let $V$ be an irreducible $FG$-module where $F$ is a field with characteristic not dividing $|G|$. Let $H\leq G$ be minimal such that there exists an $FH$-submodule $W$ of $V$ such that $W^G\cong V$. 
Let $N$ be a normal subgroup of $H/\C_H(W)$. Then $W_N$ is homogeneous. If in addition $N$ is abelian, then it is cyclic.
\end{lemma}

\begin{proof} Write $L=H/\C_H(W)$. Consider $W$ as a faithful irreducible $FL$-module. By Clifford's Theorem, if $\{W_i\}_{i=1}^r$ are the homogeneous components of $W_N$ and $I=I_{L}(W_1)$, then 	
	\[r=[L:I], \quad W_N=W_1\oplus \cdots \oplus W_r\]
and $W_1$ is an $FI$-module with $\Ind^L_I(W_1)\cong W$. 
Let $J$ be the preimage of $I$ in $H$ under the natural morphism $H\to L$. 
Then $\Ind_{J}^{H}(W_1)\cong W$ as $FH$-modules and, as a consequence, $\Ind_{J}^{G}(W_1)\cong V$. By minimality, $J=H$ and $r=1$. It follows that $W_N=W_1$ is homogeneous and so $W_N\cong U^a$ for some irreducible $FN$-modulo $U$ and $a\in \NN$. Since $W_N$ is faithful and homogeneous, $U$ is a faithful irreducible $FN$-module. If $N$ is abelian, then it is necessarily cyclic.
\end{proof}

\begin{lemma} \label{GSylowp}
	Let $p$ be a prime integer and let $G$ be a finite non-perfect group generated by elements of order $p$ with a Sylow subgroup $\GEN{a}$ of order $p$.
	Then 
	\begin{enumerate}
		\item\label{GpG'} $G=G'\rtimes \GEN{a}$ and $G'$ is a $p'$-group.
		\item\label{GpB} $G'=\GEN{b\in G : p\nmid |b| \text{ and } |ab|=p}$.
		\item\label{GpAbelian} If $G'$ is abelian, then $a$ acts on $G'$  by conjugation without fixed points. 		
	\end{enumerate}
\end{lemma}
\begin{proof}
Let $B=\{b\in G : p\nmid |b| \text{ and } |ab|=p\}$. 

\eqref{GpG'} 
	Since $G/G'$ is a non-trivial abelian group generated by elements of order $p$ whose Sylow $p$-subgroups have order at most $p$, it follows that $G/G'$ has order $ p $, $G'$ is a $p'$-group  and $G=G'\rtimes \GEN {a}$. 	
	In particular, $B\subseteq G'$.

\eqref{GpB}	We claim that $\GEN{B}$ is normal in $G$. To prove this, it suffices to show that the conjugates of elements of $B$ lie in $\GEN{B}$. If $x\in G$ and $b\in B$, then $(ab)^x= a^x b^x= a[a,x] b^x$ which has order $p$  and $[a,x]b^x\in G'$, as $b\in G'$. Thus $p\nmid |[a,x]b^x|$ and  so $[a,x] b^x\in B$. Moreover $a^x= a[a,x]$ has order $p$,  and as $[a,x]\in G'$, we get $p\nmid |[a,x]|$. Hence  $[a,x]\in B$ and $b^x\in \GEN{B}$.
	
Clearly $G/\GEN{B}$ is generated by the images of elements of order $p$ of $G$. 
	If $x$ is an element of order $p$ in $G$, then there is $i\in \{1,\dots,p-1\}$ and $y\in G$ such that $x^i=a^y=a[a,y]$. Set $b=[a,y]\in G'$, so that $p\nmid |b|$ and $|ab|=|x|=p$. Therefore $b\in B$. 
	As $G$ is generated by elements of order $p$, this proves that $G/\GEN{B}$ is generated by $a\GEN{B}$ and hence $G/\GEN{B}$ is cyclic of order $p$ and $G'=\GEN{B}$.

\eqref{GpAbelian} Finally, suppose that $G'$ is abelian. All the elements $b$  of $B$ satisfy that $ b^{a^{p-1}}b^{a^{p-2}}\dots b^a b= (ab)^p=1$. 
	Since $G'$ is abelian, every element in $G'$ satisfies the same identity. Thus, if $x\in G'$ is fixed by  conjugation with  $a$, then $1=x^{a^{p-1}}x^{a^{p-2}} \dots x^a x= x^p$, so $x=1$ because $x$ is a $p'$-element. 
\end{proof} 

\begin{lemma}\label{InducedFPF}
	Let $G=V\rtimes S$ be a group such that $V$ is a non-trivial elementary abelian $p$-group and let $q$ be a prime different from $p$. 
	Suppose that $G$ has not elements of order $pq$. Consider $V$ as an $\F_pS$-module and let $H$ be a subgroup of $S$ such that $V=\Ind_H^S(W)$ for some $\F_pH$-submodule $W$ of $V$. Then $H$ contains all the elements of order $q$ of $S$.
%
%
\end{lemma}

\begin{proof}
Let $s$ be an element of order $q$ in $S\setminus H$ and $w\in W\setminus \{1\}$. Then $1,s,\dots,s^{q-1}$ are different elements of a right transversal of $H$ in $S$. Hence $v=ww^s\dots w^{s^{q-1}}\in V\setminus \{1\}$ and $v$ commutes with $s$. So $G$ has an element of order $pq$, a contradiction.
\end{proof}

\begin{lemma}\label{Fittingp}
	Let $G$ be a finite solvable group and let $b$ be an element of $G$ of prime order $p$. If $|F(G)|$ is not multiple of $p$, then there is $x\in F(G)\setminus \{1\}$ such that $|bx|=p$.
\end{lemma}

\begin{proof}
Let $F=F(G)$. By assumption $p$ does not divide the order of $F$ and hence $b\not\in \C_G(F)$. Therefore there is $y\in F$ such that $(b,y)\ne 1$. If $|by|=p$, then taking $x=y$, the result follows. Otherwise $|by|=pn$ with $n$ a divisor of $|F|$. So $p\nmid n$ and $by=cz$ with $|c|=p$, $|z|=n$ and $(c,z)=1$. Then $z^p=(by)^p\in F$, and hence $z\in F$. Therefore $c=byz^{-1}$ has order $p$. As $(b,y)\ne 1=(c,z)$, we have $y\neq z$ and it follows that $x=yz^{-1}$ satisfies the desired condition. 
\end{proof}

\begin{lemma}\label{auxiliarpq}
	Let $p$ and $q$ be two different primes, let $G$ be a finite $\{p,q\}$-group with a Sylow $p$-subgroup $\GEN{b}$ of order $p$. If $x$ and $y$ are   $q$-elements  with $y$ nontrivial such that $bx$ and $by$ have order $p$, then  $bx$ does not commute with $y$. 
\end{lemma}\begin{proof}
	Suppose by contradiction that $bx$ commutes with $y$. Let $P$ be the subgroup of $G$ generated by all its $p$-elements. Then, by \Cref{GSylowp}, $P=Q\rtimes \GEN{b}$, where $Q$ is the Sylow $q$-subgroup of $P$. Let $i$ be the unique integer such that $y\in \Z_{i+1}(Q)\setminus \Z_{i}(Q)$, and take the quotient $\bar P= P/\Z_i(Q)= (Q/\Z_{i}(Q))\rtimes \GEN{\bar b}$. Then $\bar y=\bar y^{\bar b \bar x}=\bar y ^{\bar b}$, because $\bar y$ is central in $\bar Q$, and hence so is $\bar y^{\bar b}$. Thus $pq$ divides the order of $\bar b \bar y$, a contradiction.
\end{proof}

We finish this section with a recent result from \cite{Debon} which will be used in \Cref{Section-t}. 

\begin{theorem}\label{VS2} Let $p\ge 5$ prime, let $G$ be a non-trivial rational finite $2$-group and let $V$ be a finite dimensional faithful $\F_pG$-module with the eigenvector property. Then $p=5$ and there is a $2$-subgroup $K$ of the symmetric group $\Sigma_n$ such that $G\cong Q_8\wr K$, $\dim V=2n$ and  $V\rtimes G\cong \MM\wr K$.
\end{theorem}

\section{Technical results for \cut groups}\label{SectionTechnicalCut}

In this section we collect technical lemmas about \cut groups. We start collecting in the following lemma the statements of Lemmas~2.6, 3.4, 3.5 and 4.3 in \cite{BKMdR}.

\begin{lemma}\label{CutOrders1}
Let $G$ be a finite $\cut$ group and let $p$ and $q$ be two different prime divisors of $|G|$.
\begin{enumerate}
	\item\label{5rational} If $g$ is a $p$-element of $G$ and $p\equiv 1 \mod 4$, then $g$ is rational in $G$.
	\item\label{Gpexponentpo4} If $G_p$ is abelian, then the exponent of $G_p$ divides $p$ or $4$.
	\item\label{cyclicsylow} If $G_p$ is normal in $G$ and $G$ does not have an element of order $pq$, then $G_q$ is either quaternion of order $8$ or cyclic of order dividing $q$ or $4$. 
	\item\label{noelements4p} If $q=2$ and $G_2$ is cyclic, then $G$ has no elements of order $4p$. 
	\item\label{noelements2p} If $q=2$, $p\equiv 1 \mod 4$ and $G_2$ is either cyclic or quaternion of order $8$, then $G$ has no elements of order $2p$.
	\item\label{noelements37} If $G_3$ is cyclic, then $G$ has not elements of order $3\cdot 7$.
\end{enumerate}
\end{lemma}

\begin{lemma}\label{G2=Q8} Let $G$ be a finite \cut group with    a Sylow subgroup $G_2$ isomorphic to $Q_8$. Let $a\in G$ and let $p$ be an odd prime integer.
	\begin{enumerate}
		\item\label{4pnotrational} If the order of $a$ is $4p^n$ with $n\ne 0$, then every $2$-element of $\N_G(\GEN{a})$ commutes with $a_p$. 
		\item\label{No357-437} The order of $a$ is neither $3\cdot 5\cdot 7$ nor $4\cdot 3\cdot 7$.
		\item\label{Q8commutes} If $V$ is a normal $p$-subgroup of $G$, then $\C_{G_2}(V)$ is either $1$ or $G_2$.
	\end{enumerate}
\end{lemma}

\begin{proof} 
	\eqref{4pnotrational} 
By means of contradiction, suppose that  $|a|=4p^n$ with $n\ge 1$ and  $b$ is a $2$-element of $\N_G(\GEN{a})$ not commuting with $a_p$.   In particular, $b\in \N_G(\GEN{a_2})\setminus \GEN{a_2}$, so  $\GEN{a_2,b}\cong Q_8$ and $a_p^b\in \GEN{a_p}\setminus \{a_p\}$. Moreover, $a_p^{b^2}=a_p^{a_2^2}=a_p$, so $a_p^b=a_p^{-1}$. Then  $a^b=a^{-1}$ and   $G$ is \cut, so  $a$ is rational in $G$. Hence, there exists a $2$-element $c\in G$ such that $a_2^c=a_2$ and $a_p^c=a_p^{-1}$. The latter implies that $c\not \in \GEN{a_2}$, while the former implies that $\GEN{a_2,c}=\GEN{a_2}$, since it is an abelian subgroup of $Q_8$ and $|a_2|=4$. This yields the desired contradiction.

	\eqref{No357-437} If $|a|=3\cdot 5\cdot 7$, then $\N_G(\GEN{a})/\C_G(a)$ is isomorphic to a subgroup of index at most $2$ in $C_2\times C_4\times C_6$ while a Sylow $2$-subgroup of $\N_G(\GEN{a})/\C_G(a)$ must be a subquotient of $Q_8$, which is not possible.
	
	Suppose that $|a|=4\cdot 3\cdot 7$. As $a$ is inverse semi-rational in $G$, there exists $g\in G$ such that $(a^5)^g=a$ or $(a^5)^g=a^{-1}$. The former implies $a_2^g=a_2$ and $a_3^g=a_3^{-1}$. Hence, $\GEN{a_2,g_2}$ is isomorphic to an abelian subgroup of $Q_8$ and $|a_2|=4$. Thus $\GEN{a_2,g_2}=\GEN{a_2}\cong C_4$. This is not possible since $a_2$ and $a_3$ commute while $a_3^g=a_3^{-1}$ and hence $a_3^{g_2}=a_3^{-1}$. 
	Then $(a^5)^g=a^{-1}$ and so  $a_2^g=a_2^{-1}$, $a_3^g=a_3$ and $a_7^g=a_7^{4}$. Hence, $\GEN{a_2,g_2}\cong Q_8$ with $a_3^{g_2}=a_3$ and $a_7^{g_2}=a_7$. In particular, $\C_G(a_3a_7)$ contains a Sylow $2$-subgroup of $G$, a contradiction, since $a_3a_7$ is semi-rational in $G$. 
	
	\eqref{Q8commutes} 	
	Suppose that    $\C_{G_2}(V)$ is neither $1$ nor $G_2$. 
	Consider the series $V_0=V> V_1> \dots > V_k=1$ with $V_i=\Phi(V_{i-1})$, the Frattini subgroup of $V_{i-1}$. Note that $V_i$ is normal in $G$, as so is $V$ and $V_i$ is characteristic in $V$.
	Let $i$ be maximal with   $G_2\ne \C_{G_2}(V_i)$. 
	Then $i<k$ and $G_2\subseteq \C_G(V_{i+1})$. 
	Moreover, conjugation induces an action of $G$ on $W=V_i/V_{i+1}$. As $W$ is $p$-elementary abelian, we may consider $W$ as an $\F_p G$-module.
	Since $\C_{G_2}(W)\ne 1$, $W$ is an $\F_p(C_2\times C_2)$-module and hence, each simple $\F_pG_2$-module of $W$ is linear. 
	We claim that $G_2\subseteq \C_G(W)$. Otherwise,
	$W$ has a simple submodule $M=\GEN{mV_{i+1}}$ such that $\C_{G_2}(M)$ is cyclic of order $4$, say generated by $x$. Then the image $g$ of $xm$ in $G/V_{i+1}$ has order $4p$. Since the image of $G_2$ in $G/V_{i+1}$ normalizes $\GEN{g}$ we obtain the contradiction $G_2\subseteq \C_{G}(M)$, by  \eqref{4pnotrational}.
	So $G_2\subseteq \C_G(W)$. Therefore, for every $v\in V_i$ and every $g\in G_2$, $v^g=vz$ for some $z\in V_{i+1}$. Recall that $G_2\subseteq \C_G(V_{i+1})$. Then $v=v^{g^4}=vz^4$, so that $z^4=1$ and so $z=1$. Thus $G_2\subseteq \C_G(V_i)$, a contradiction.
\end{proof}

\begin{lemma}\label{V5or7}
Suppose that $G$ is a finite solvable \cut group such that $\GK(G)$ has four edges and for every non-trivial normal subgroup $N$ of $G$ we have $\GK(G/N)\ne \GK(G)$. Let $V$ be a non-trivial normal $p$-subgroup of $G$. Then $p\in \{5,7\}$ and $V$ is an elementary abelian Sylow $p$-subgroup of $G$. In particular, $V$ is the unique non-trivial normal $p$-subgroup of $G$.
\end{lemma}

\begin{proof}
	Without loss of generality we may assume that $V$ is a minimal normal $p$-subgroup of $G$. Then $V$ is an elementary abelian $p$-group. By assumption, $\GK(G/V)$ is properly contained in $\GK(G)$. If $\pi(G)=\pi(G/V)$, then $\GK(G/V)$ has less than four edges which is not possible by the main theorem of \cite{BKMdR}. Thus $\pi(G) \ne \pi(G/V)$ and so $V$ is a Sylow $p$-subgroup of $G$. Moreover, $p\in \{2,3,5,7\}$, by \cite{Bachle2018}, and $\GK(G/V)$ has three vertices. Then $\GK(G/V)$ contains the edge $2-3$ by \cite[Corollary~B]{BKMdR}. So $p\in \{5,7\}$.
\end{proof}

	\begin{lemma}\label{G2C2ornoelements21}
	Let $G$ be a finite solvable \cut group. If  $\pi(G)=\{2, 3, 7\}$ and $G_2$ and $G_7$ are cyclic, then either $G_2\cong C_2$ or $G$ has no elements of order $3\cdot 7$.
\end{lemma}
 
\begin{proof}
By means of contradiction, suppose that  $\pi(G)=\{2,3,7\}$, $G_2$ and $G_7$ are cyclic,  $G_2\not \cong C_2$ and $G$ has an element $g$ of order $3\cdot 7$. 
Let $N = \N_G(\GEN{g^3})$ and $H=\C_G(g^3)$.
By \Cref{CutOrders1}\eqref{Gpexponentpo4}, $G_2\cong C_4$ and $G_7\cong C_7$, so $|G_{\{2,7\}}|=28$ and, by the Sylow Theorem, $G_7 \trianglelefteq G_{\{2,7\}}$. 
Hence, $N$ contains an element of order $4$.
On the other hand, it follows from \Cref{CutOrders1}\eqref{noelements4p}, that $H$ has no element of order $4$.
Using that $g^3$ is inverse semi-rational it follows that $N/H\cong C_6$, so that $H$ has an element of order $2$ and hence $H_{3,7}$ is of index $2$ in $H$. 
This means that $[N : H_{3,7}] = 12$ and $H_{3,7}\trianglelefteq H$.
Since $H_{3,7}=\GEN{g^3}\times H_3$, $H_3$ is characteristic in $H_{3,7}$, and hence, $H_3$ is normal (and characteristic) in $H$. 
As $H$ is normal in $N$, $H_3\trianglelefteq N$. 
Since also $\GEN{g^3}\trianglelefteq N$, we have that $H_{3,7} \trianglelefteq N$ and $K:=N/H_{3,7}$ is $C_{12}$, $C_2 \times C_6$, $A_4$, $C_3\rtimes C_4$ or $D_{12}$.
By \Cref{CutOrders1}\eqref{noelements4p}, $G$ has no elements of order $12$ and hence $N/H_{3,7}$ is not $C_{12}$. 
As $N$ has an element of order $4$, $N/H_{3,7}$ is neither $C_2\times C_6, A_4$ nor $D_{12}$.
Thus $K\cong C_3 \rtimes C_4$, and hence $H/H_{3,7}$ is the unique subgroup of order $2$ of $K$ and  
	\[N/H=K/(H/H_{3,7})\cong S_3,\] 
in contradiction with $N/H\cong C_6$.
\end{proof}

\begin{lemma} \label{S3elementaryabelian} 
The following conditions are equivalent for a finite group $G$:
\begin{enumerate}
\item $G$ is solvable, \cut, has a normal Sylow subgroup of order $5$, and its GK-graph is contained in $(2-3-5)$.
\item $G=(\GEN{a}\times C)\rtimes \GEN{b}$ with $|a|=5$, $|b|=4$, $a^b=a^2$,  $C$ is elementary abelian $3$-group (possibly trivial) and $c^b=c^{-1}$ for every $c\in C$.
\end{enumerate}	
\end{lemma}

\begin{proof}
Verifying that a group as given in (2) satisfies the conditions in (1) is straightforward.

Suppose that $G$ satisfies the conditions in (1). 
Let $a\in G$ of order $5$. As $2-5\not \in \GK(G)$ and an element of order $5$ in a \cut group is rational (\Cref{CutOrders1}\eqref{5rational}), $\C_G(a)=\GEN{a}\times C$ with $C$ a $3$-group and $G=\N_G(\GEN{a})=(\GEN{a}\times C)\rtimes \GEN{b}$ where $b\in G$ has order $4$ and $a^b=a^2$. 

Suppose by contradiction that $c$ is an element of $C$ of order $9$. Then it commutes with $a$ and since $G$ is \cut, there exists $s\in G$ such that $(a^2c^2)^s$ is either $ac$ or $a^{-1}c^{-1}$. In any case, $12$ divides the order of $s$ in contradiction with \Cref{CutOrders1}\eqref{noelements4p}.
Therefore the exponent of $C$ divides 3.

Let $c\in C$ and $c_0=c^bc$. Then $(bc)^2 = b^2 c_0$ and $bc$ has order $4$, since $G$ has no elements of order $12$. So $b^2 c_0$ has order $2$ and hence $c_0^{b^2}c_0=1$. By the inverse semi-rationality of $ac_0$ in $G$, there is $x\in \N_G(\GEN{c_0})$ of order $4$, which may be taken of the form $x=b c_1$ with $c_1\in C$ because $a$ centralizes $C$. Then $c_0=c_0^{x^2}= c_0^{b^2 c_1^b c_1}= (c_0^{-1})^{c_1^b c_1}$. If $c_0\ne 1$, then  $c_1^b c_1$ is a $3$-element acting on $\GEN{c_0}$ with order $2$, a contradiction. So $c_0=1$. This proves that $c^b=c^{-1}$ for every $c\in C$.
Thus the inversion map is an automorphism of $C$ and hence $C$ is abelian.
\end{proof}

\begin{lemma}\label{rationalS2} Let $G$ be a finite solvable \cut group with a normal Sylow subgroup of order $7$ and $\GK(G)\subseteq (3-2-7)$. Then $G=\GEN{a}\rtimes (\GEN{b}\times Q)$ with $|a|=7$, $|b|=3$, $a^b=a^2$ and $Q$ a rational $2$-group.
\end{lemma}

\begin{proof}
By hypothesis $G$ has an element $a$ of order $7$ such that $\GEN{a}$ is a normal Sylow $7$-subgroup of $G$. 
Then a Sylow $3$-subgroup of $G$ has order $3$, by \Cref{CutOrders1}\eqref{cyclicsylow}.  
As $a$ is inverse semi-rational in $G$, $G$ has an element $b$ of order $3$ such that $a^b=a^2$.
Let $H$ be a Hall $\{2,3\}$-subgroup of $G$ containing $b$ and let $Q$ be a Sylow $2$-subgroup of $H$. 
Then $G=\GEN{a}\rtimes H$ and $H=\GEN{Q,b}$. 

We prove first that $\GEN{b }$ is normal in $H$. Otherwise, by \Cref{Fittingp}, there is $y\in Q\setminus \{1\}$ with $|by|=3$. Then $y$ centralizes $a$, for otherwise, $a^y=a^{-1}$ and then $a^{by}=a^{-2}$ and, as $by$ has order $3$, $1\equiv (-2)^3\equiv -1 \mod 7$, a contradiction. 
Since $G$ is \cut, there is an element $g\in G$, such that $(ay)^g=a^2y$ or $(ay)^g=a^{-2}y^{-1}$. 
Suppose that $(ay)^g=a^2y$. As $(a,y)=1$ and the automorphism of $\GEN{ay}$ mapping $ay$ to $a^2y$ has order $3$, we may assume without loss of generality that $g$ is an element of order $3$ in $H$. 
Then $g=b^ix$ with $x\in Q$ and $i\in \{1,2\}$. However, if $i=2$, then 
$a^2=a^{b^2x}=(a^4)^x$, but this contradicts the fact that $x$ is a $2$-element. So $(ay)^{bx}=a^2y$. In particular $bx$ commutes with $y$, a contradiction with \Cref{auxiliarpq}.
Therefore $(ay)^g=a^{-2}y^{-1}$ and then $(ay)^{g^2}=a^4y$. Arguing as in the previous case, we deduce that $H$ contains and element $h$ of order $3$ such that $(ay)^h=a^4y$. So $h=b^ix$ with $i\in \{1,2\}$ and $x\in Q$. If $i=1$, then $a^4=a^{bx}=(a^2)^x$, contradicting the fact that $x$ is a $2$-element. So $i=2$. Observe that $x$ belongs to the group $P$ generated by the elements of $H$ of order $3$. 
By \Cref{GSylowp}, $P=P'\rtimes \GEN{b}$ and $P'$ is a Sylow $2$-subgroup of $P$. Hence $x\in P'$. Then $z=x^{b^2}x\in P'$, so it is a $2$-element and as $h^2=bz$, $bz$ has order $3$ and $(ay)^{bz}=a^2y$. 
So, $bz$ commutes with $y$, in contradiction with \Cref{auxiliarpq}.
Therefore $\GEN{b}$ is normal in $H$.

If $b$ is not central in $H$, then there is an element $q\in Q$ such that $b^q=b^2$, but the automorphism group of $\GEN{a}$ is abelian, so $a^2=a^b= a^{q^{-1} b q}= a^{b^q}=a^{b^2}=a^4$, a contradiction. Thus $H=\GEN {b}\times Q$, as desired. 
Finally, as $\GEN{b}\times Q\cong G/\GEN{a}$ is \cut,  we derive that $Q$ is rational. Indeed,  suppose that $q\in Q$ is not rational in $Q$. We have that  $\N_{\GEN{b}\times Q}(\GEN{bq})= \GEN{b}\times \N_Q(\GEN{q})$ and $\C_{\GEN{b}\times Q}(\GEN{bq})= \GEN{b}\times \C_Q(\GEN{q})$, so 
 $$| \N_{\GEN{b}\times Q}(\GEN{ bq}):\C_{\GEN{b}\times Q}(bq)|= |\N_Q(\GEN{q}): \C_Q(q)|=\Aut(\GEN{q})/2 <\Aut(\GEN{q})= \Aut(\GEN{bq})/2, $$
 a contradiction because $\GEN{b}\times Q$ is \cut.
\end{proof}
	
\section{Graph (s)}\label{Section-s}
	
In this section we prove that the graph (s) in \Cref{QuestionE} is not the GK-graph of a finite solvable \cut group. By means of contradiction we fix a finite solvable \cut group $G$ of minimal order  with 
\begin{equation}\label{Grafo-s}
	\GK(G)= \begin{tikzpicture}[local bounding box=bb,baseline=(bb.center)]
		\node[label=west:{$2$}] at (0,0.5) (2){};
		\node[label=east:{$3$}] at (0.5,0.5) (3){};
		\node[label=west:{$5$}] at (0,0) (5){};
		\node[label=east:{$7$}] at (0.5,0) (7){};
		\foreach \p in {2,3,5,7}{
			\draw[fill=black] (\p) circle (0.075cm);
		}
		\draw (2)--(3);
		\draw (2)--(7);
		\draw (3)--(5);
		\draw (3)--(7);
	\end{tikzpicture}\ 
\end{equation} 

We start fixing a minimal normal subgroup $V$ of $G$. 
The minimality hypothesis on $G$ implies that it satisfies the hypothesis of \Cref{V5or7}, so $V$ is an elementary abelian Sylow $p$-subgroup of $G$ with $p\in\{5,7\}$. 
This shows that $\pi(F(G))\subseteq \{5,7\}$. 
However, $G$ does not have elements of order $5\cdot 7$, so that $F(G)=V$.
We fix also a Hall $p'$-subgroup $S$ of $G$. 
Then $G=V\rtimes S$ and $S$ acts irreducibly on $V$, since $V$ is minimal normal in $G$. 
Furthermore, $\C_S(V)\subseteq V\cap S=1$ since $F(G)=V$ and $G$ is solvable. Therefore $\C_S(V)=1$, i.e. the action of $S$ on $V$ is faithful. 

\begin{fact}\label{F7elementaryabelian}
$p=7$.
\end{fact}

\begin{proof}
By the arguments above, we only have to prove that $p\ne 5$. 
By means of contradiction assume otherwise. 
In view of \Cref{CutOrders1}.\eqref{cyclicsylow}, $G_7\cong C_7$ and $G_2$ is $Q_8$ or cyclic of order divisor of $4$. Applying \Cref{G2C2ornoelements21} to $S$ we get that $G_2$ is $C_2$ or $Q_8$. 
By \Cref{CutOrders1}\eqref{5rational}, every element of order $5$ of $G$ is rational in $G$, so $G_2\not \cong C_2$ and hence $G_2\cong Q_8$.

We claim that $G_{\{2,7\}}\cong C_7 \times Q_8$. For this, we first prove that $G_{\{2,7\}}$ has a normal Sylow $7$-subgroup. 
Otherwise, by the Sylow Theorem, it has $8$ subgroups of order $7$ and hence it has $6\cdot 8$ elements of order $7$, so that the remaining $8$ elements form the unique Sylow $2$-subgroup. 
Thus, $G_{\{2,7\}}=Q_8\rtimes C_7$. 
However, $Q_8$ does not have automorphisms of order $7$ and therefore $C_7$ is central, and in particular, normal in $G_{\{2,7\}}$. 
Thus $G_{\{2,7\}}$ is either $Q_8\times C_7$ or $C_7\rtimes Q_8$. 
As $|\Aut(C_7)|=6$, any element of order $4$ in $Q_8$ acts as the inversion or the identity in $C_7$. 
In any case, there is an element of order $4$ of $Q_8$ acting trivially on $C_7$. We conclude by \Cref{G2=Q8}\eqref{4pnotrational} that $G_{\{2,7\}}=Q_8\times C_7$, as desired.

Fix an element $a\in G$ of order $7$.
We now analyze the structure of $\C_G(a)$.
Note firstly that, as $G$ has an element of order $3\cdot 7$ and $\GEN{a}$ is a Sylow $7$-subgroup of $G$, $\C_G(a)$ contains elements of order $3$. 
If $\C_G(a)$ has an element $b$ order $3\cdot4$, then $x=ab$ has order $3\cdot4\cdot7$, a contradiction with \Cref{G2=Q8}\eqref{No357-437}. This proves that $\C_G(a)$ does not have elements of order $12$.

We claim that for every $3$-element $b\in \C_G(a)$, there is a $2$-element $c\in \C_G(a)$, such that $b^c=b^{-1}$.
Indeed, let $b\in \C_G(a)$ of order $3^r$ for some $r\geq 1$, and let $x=ab$. 
As $x$ is inverse semi-rational in $G$, there is a $2$-element $c\in \N_G(\GEN{x})\setminus \C_G(x)$. 
However, every $2$-element of $\N_G(\GEN{x})$ belongs to $\N_G(\GEN{a})$ and hence also
belongs to $\C_G(a)$, since $G_{\{2,7\}}\cong Q_8\times C_7$. Therefore $c\not \in \C_G(b)$ and  hence  $b^c=b^{-1}$.
We have proved that $\C_G(a)_{\{2,3\}}$ satisfies the hypothesis of \Cref{FabelianSylow}. 
Then, $F(\C_G(a)_{\{2,3\}})$ is an abelian Sylow $3$-subgroup and $b^c=b^{-1}$ for every element $c\in \C_G(a)_{\{2,3\}}$ of order $2$ and every $b\in F(\C_G(a)_{\{2,3\}})$. 

Let $c\in \C_G(a)$ of order $2$ and $w\in V\setminus\{1\}$. If $v=ww^c$, then $v^c=v$ and $v=1$ as $G$ does not contain elements of order $2\cdot 5$. Hence, $w^c=w^{-1}$ for every $w\in V\setminus\{1\}$.  As the action of $S$ on $V$ is faithful, we deduce that $c\in Z(S)$. This is a contradiction with $F(\C_G(a)_{\{2,3\}})$ being a $3$-group.
\end{proof}

So we have $G=V\rtimes S$, with $V$ an elementary abelian Sylow $7$-subgroup of $G$ and $S$ a Hall $7'$-subgroup of $G$ acting faithfully and irreducibly on $V$. Clearly $\GK(S)= (2-3-5)$. Moreover, by \Cref{CutOrders1}\eqref{cyclicsylow}, the Sylow $5$-subgroups of $S$ have order $5$.

\begin{fact}\label{C5notnormal} The Sylow $5$-subgroups of $S$ are not normal in $S$.
\end{fact}

\begin{proof} 
By means of contradiction suppose that $S$ has a normal Sylow $5$-subgroup. 
By \Cref{S3elementaryabelian}, $S\cong (\GEN{a}\times E)\rtimes \GEN{b}$ with $|a|=5$, $|b|=4$, $a^b=a^2$, $E$ elementary abelian $3$-group and $e^b=e^{-1}$ for every $e\in E$.   
Consider $V$ as an $\F_7S$-module and let $F$ be a splitting field for $S$ containing $\F_7$. 
Then 
$$V\otimes_{\F_7} F\cong W_1\oplus \cdots \oplus W_k$$
with $W:=W_1,\dots,W_k$ absolutely irreducible $FS$-modules. As $V_{\F_7S}$ is faithful, so is each $W_{i}$, because, since $V_{\F_7 S}$ is irreducible, the characters of the $W_i$'s form a Galois orbit of that of $W$ \cite[Theorem~9.21]{Isaacs1976}. Let $\varphi$ be the Brauer character afforded by $W$. Hence $\varphi\in\Irr(S)$ and $\varphi$ is faithful. 
By \cite[Theorem~3.5.12]{JespersdelRioGRG1}, $\varphi=\lambda^S$ for a linear character $\lambda$ of a normal subgroup $H$ of $S$ containing $\GEN{a}\times E$,  and, if $K$ is the kernel of $\lambda$, then $H/K$ is a cyclic maximal abelian subgroup of $\N_S(K)/K$.
Then $1=\ker \varphi = \Core_S(K)$. 
This implies that $K\cap S_{2'}=1$ as every subgroup  of $S_{2'}$ is normal in $S$. 
Then $K$ is a $2$-subgroup of $S$. Assume $K\neq 1$ and let $x\in K$ of order $2$.
Then  $x=yb^2$ with $y\in \GEN{a}\times E\subseteq \C_S(a)$. Hence $a^x=a^{b^2}=a^4=a^{-1}$ and so  $a^{2}=(x,a)\in K$, a contradiction. Thus $K=1$, i.e. $\lambda$ is a faithful character of $H$. So $H$ is cyclic, and as it contains $E$ it follows that $E$ is a cyclic Sylow $3$-subgroup of $S$. Thus $G_3$ is cyclic and hence $G$ has no elements of order $3\cdot 7$ by \Cref{CutOrders1}.\eqref{noelements37}, in contradiction with \eqref{Grafo-s}. 
\end{proof}
	
In the remainder of the proof $a$ is an element of $S$ of order $5$ and $C=\C_S(a)_{5'}$.

\begin{fact}\label{NormalizerS5}
$C$ is a $3$-group, $\N_S(\GEN{a})$ has an element of order $4$ and for every element $b$ of order $4$ in $\N_S(\GEN{a})$, 
	\[\N_S(\GEN{a})=(\GEN{a}\times C) \rtimes \GEN{b}\qand a^b=a^{\pm 2}.\] 
Moreover, $\N_S(\GEN{x})=\C_S(x)$ for every $2$-element $x\in \N_S(\GEN{a})$.
\end{fact}

\begin{proof}
As $a$ is rational in $S$, $\N_S(\GEN{a})/\C_S(a)$ is cyclic of order $4$ and hence $\N_S(\GEN{a})=(\GEN{a}\times C)\rtimes \GEN{b}$ with $C$ a $3$-group and $b$ an element of order $4$ of $S$ such that $a^b=a^2$. If $b_1$ is another element of order $4$ of $\N_S(\GEN{a})$, then $b_1=xb^{\pm 1}$ for some $x\in \C_S(a)$ and hence $\N_S(\GEN{a})=(\GEN{a}\times C)\rtimes \GEN{b_1}$ and $a^{b_1}=a^{\pm 2}$.

Let $N$ be maximal among the normal subgroups of $S$ not containing $a$. Then $\GEN{aN}$ is normal in $S/N$, for otherwise $F(S/N)=M/N$ with $M$ a normal subgroup of $S$ not containing $a$ and containing $N$ properly. By \Cref{S3elementaryabelian}, we have $S/N\cong C_5\rtimes C_4$. Thus every element of order $4$ in $S/N$ is not conjugate to its inverse and $N\cap \N_S(\GEN{a})=C$ (because $S=N\N_S(\GEN{a})$, by the Frattini argument). Let $x$ be a $2$-element of $\N_S(\GEN{a})$. If the order of $x$ is less than $4$, then clearly, $\N_S(\GEN{x})= \C_S(x)$. Otherwise $|x|=4$ and hence, $|xN|=4$. Moreover, $xN$ is not conjugate to $x^{-1}N$ in $S/N$ and hence $x$ is not conjugate to $x^{-1}$ in $S$ so that, again, $\N_S(\GEN{x})= \C_S(x)$.
\end{proof}

\begin{fact}\label{Strivialcenter}
$Z(S)=1$.
\end{fact}
\begin{proof}
Assume that $z$ is a central element of $S$ of prime order. Since  $S$ is \cut and  $\GK(S)= (2-3-5)$, the order of $z$ is $3$. By \Cref{NormalizerS5},  $\N_S(\GEN{a})$ has an element $b$ of order $4$ and $\N_S(\GEN{b})=\C_S(b)$. Then $bz$ has order $12$ and $\N_S(\GEN{zb})=\C_S(zb)$, so $zb$ is not inverse semi-rational in $S$, a contradiction.
\end{proof}

Let $X$ be the subgroup generated by the elements of order $5$ of $S$.
By \Cref{GSylowp}, we have $X=X'\rtimes \GEN{a}$ with $X'$ a $5'$-group, and, by \Cref{C5notnormal}, $(X',a)\ne 1$. In particular, $X'$ is a non-trivial normal subgroup of $S$ and $\pi(X')\subseteq \{2,3\}$.

\begin{fact}\label{X35}
$\pi(X')=\{3\}$ and $F(X)=X'\ne 1$.
\end{fact}

\begin{proof}
By the arguments above, it suffices to show that $|X'|$ is odd.
	Suppose by contradiction that $2$ divides the order of $X'$. 
	Let $\mathcal S_2$ be the set of Sylow $2$-subgroups of $X'$. Then $|\mathcal S_2|$ is a power of $3$. Moreover $\GEN{a}$ acts by conjugation on $\mathcal S_2$, and therefore  normalizes  at least one Sylow $2$-subgroup, say $L$.     So  $\GEN{a}$ acts  by conjugation on  $L$ (without fixed points, because there are no elements of order $10$ in $S$). Let $L_0=\Omega(\Z(L))$, which is elementary abelian and nontrivial. 
	Then $\GEN{a}$ also acts by conjugation on $L_0$ without fixed points, so $L_0$ is a direct sum of simple $\F_2\GEN{a}$-modules of dimension $4$.  Therefore, if we take $e\in  L_0\setminus \{1\}$, generating one of them, then   $\GEN{e,a}\cong C_2^4\rtimes_{\text{Fr}} C_5$. Now we see $V$ as a $\F_7\GEN{e,a}$-module.
	As the action of $S$ on $V$ is faithful, so is the action of $\GEN{e,a}$. Then, by \Cref{FrobeniusFaithful}, $\GEN{a}$ does not act fixed-point-freely on $V$, which gives the desired contradiction.		
\end{proof}

Let $Q=X'C$. Combining the Frattini argument with \Cref{NormalizerS5} and \Cref{X35} we get
	\begin{equation}\label{S-s}
		S=X\N_S(\GEN{a})=(Q\rtimes \GEN{a})\rtimes \GEN{b}, \text{ with } |b|=4, \ a^b=a^2 \text{ and } \pi(Q)=\{3\}.
	\end{equation}
In particular, $Q$ is the unique Sylow 3-subgroup of $S$.

Let $H$ be a minimal subgroup of $S$ such that there exists an $\mathbb{F}_7H$-submodule $W$ of $V$ with $W^S\cong V$. We denote by $K$ the kernel of the action of $H$ on $W$. 
By \Cref{InducedFPF}, $H$ contains all the elements of order $5$ of $S$, i.e. $X\subseteq H$, and hence $\N_S(\GEN{a})$ contains a right transversal $T$ of $H$ in $S$, which will be fixed throughout the section.
As $K$ is normal in $H$ and the action of $S$ on $V$ is faithful we get
	\begin{equation}\label{Kcorefree}
\bigcap_{t\in T}K^t=\Core_S(K)=1.
	\end{equation}

By \Cref{HAbelCic}, all abelian normal subgroups of $H/K$ are cyclic, so \Cref{StructureFED} applies for the group. In this particular case we can say more.

\begin{fact}\label{StructureHK}
Let $F=F(H/K)$ and  let  $Z$  be  the socle of $Z(F)$. Then there exist $E,D\trianglelefteq H/K$ such that:
\begin{enumerate}[(a)]
\item $F=ED$ is a $3$-group, $Z=E\cap D$ and $D=\C_F(E)=\C_{H/K}(E)$ is cyclic.
\item $E$ is cyclic of order $3$ or extra-especial of exponent $3$;
\item $Z=\GEN{z_0K}\cong C_3$ with $w^{z_0}=w^2$ for every $w\in W$. Moreover, $Z \subseteq Z(H/K)$ and it is the unique minimal normal subgroup of $H/K$. 
\end{enumerate}  
\end{fact}

\begin{proof} Let $E$ and $D$ be the groups given in \Cref{StructureFED} for $H/K$. 
Note that $H/K$ has an element of order $5$, since $5-7\not \in \GK(G)$. We claim that $|F|$ is odd. Otherwise, $|Z|$ is even and $Z_2$ is a normal cyclic subgroup of $H/K$. Hence, an element of order $5$ of $H/K$ acts trivially on $Z_2$, a contradiction with the fact that $G$ does not have elements of order $10$. Moreover, if $5\mid |F|$, then $\GEN{aK}$ is a normal subgroup of $H/K$. Therefore $\GEN{a,K}$ contains all the elements of order $5$ of $H$, and hence of $S$, so $X\subseteq \GEN{a,K}$ and $(X,\GEN{a})\subseteq (\GEN{a}K,\GEN{a})\subseteq K$. As $T\subseteq \N_S(\GEN{a})$, we have $(X,\GEN{a})=(X^t,\GEN{a}^t)\subseteq K^t$ for all $t\in T$. Then, by \eqref{Kcorefree}, $(X,\GEN{a})=1$,  and therefore, as $S=XN_S(\GEN{a})$,  $\GEN{a}$ is normal in $S$,   in contradiction with \Cref{C5notnormal}. Therefore, $F$ is a $3$-group and $Z\cong C_3$. 
Then statements (a) and (b) follow directly from \Cref{StructureFED}.		

Now we prove (c). 
Let $N$ be a minimal normal subgroup of $H/K$. Since $F$ is a $3$-group, $N$ is $3$-elementary abelian and hence, cyclic of order $3$. Since $E$ is a $3$-group, $E$ commutes with $N$ and therefore $N\subseteq \C_F(E)=D$, by (a). In particular $N=Z$, as $D$ is cyclic.
		
Consider the restriction to $E$ of the $\F_7(H/K)$-module $W$ and denote it by $W_E$. 
By \Cref{HAbelCic}, $W_E$ is homogeneous,
i.e. $W_E\cong U^l$ for some irreducible $\F_7E$-module $U$ and some $l\geq 1$. 
As $W$ is faithful, so are $W_E$ and $U$.
Since the exponent of $E$ is $3$, $\F_7$ is a splitting field for $E$ (as $\F_7$ contains a primitive $3$-th root of unity). In particular, $U$ is absolutely irreducible and a non-trivial central element of $E$ acts as the scalar multiplication on $U$ by an element of order $3$ of $\F_7$. Therefore, $Z=\GEN{z_0K}$ with $z_0K$ acting by scalar multiplication by $2$ on $U$, and hence, on $W$, as its restriction to $E$ is homogeneous. Since the action of $H/K$ on $W$ is faithful, $Z$ is central in $H/K$. 
	\end{proof}
	
	In the remainder of the proof, we use the notation of \Cref{StructureHK}.
%
%
	
	\begin{fact}\label{bnotH} $b\not\in H$.
	\end{fact}
	
\begin{proof} 
Assume by contradiction that $b\in H$. 
Recall that $Q$ is the unique Sylow $3$-subgroup of $S$ and $Q=X'C$.
Moreover, $1\ne F(X)=X'$, by \Cref{X35}.
Thus there is an element $z$ of order $3$ in $F(X)\cap Z(Q)$.
Since $Z(F(X))K/K$ is an abelian normal subgroup of $H/K$, it is a cyclic subgroup contained in $F(H/K)$. By \Cref{StructureHK}, $Z$ is the unique minimal normal subgroup of $H/K$, and hence $Z$ is the unique subgroup of $Z(F(X))K/K$ of order $3$. It follows that $\GEN{zK}=Z$ or $z\in K$. Suppose that $z\in K$. Let $s\in S$. Then $s=b^\varepsilon xc$ for some $\varepsilon\in\{0,1,2,3\}$, $x\in X$ and $c\in C$. As $C\subseteq Q$, $z\in Z(Q)$ and $b,x\in H$, we have $z^{c^{-1}}=z\in K= K^{b^{\varepsilon}x}$. Hence, $z\in K^{b^\varepsilon xc}=K^s$. This shows that $z$ is a non-trivial element of $\Core_S(K)$, in contradiction with \eqref{Kcorefree}. So $\GEN{zK}=Z$. By \Cref{StructureHK}, replacing $z$ by $z^{-1}$ if needed, we may assume that $w^z=w^2$ for every $w\in W$.

As $X\GEN{b}\subseteq H$, it follows that $S=CH$. 
Since $Z$ is central in $H/K$ and $(z,C)=1$, it follows that $z^sK=zK$ for every $s\in S$, i.e. there is $k_s\in K$ such that $z^s=zk_s$. 
Let $v\in V$ be arbitrary. Write $v=\prod_{t\in T}w_t^t$, with $w_t\in W$ for each $t\in T$. Then 
$$v^z=\prod_{t\in T}w_t^{tz}=\prod_{t\in T}w_t^{z^{t^{-1}}t}=\prod_{t\in T}(w_t^{zk_{t^{-1}}})^{t}=\prod_{t\in T}(w_t^{2})^{t}=v^2.$$
As the action of $S$ on $V$ is faithful, this implies that $z\in Z(S)$ in contradiction with \Cref{Strivialcenter}.
\end{proof}

	


\begin{fact}\label{centralizerb}
	$7\in \pi(\C_G(b))$.
\end{fact}
\begin{proof}
Assume by contradiction that $7\not\in \pi(\C_G(b))$. 
For every $w \in W$, $ww^bw^{b^2}w^{b^3}$ is a $7$-element in $\C_G(b)$ and hence $ww^bw^{b^2}w^{b^3}=1$. 
This implies that $W+W^b+W^{b^2}+W^{b^3}$ is not a direct sum and as $V=\Ind^G_H(W)$, it follows that $H,Hb,Hb^2,Hb^3$ are not pairwise different. Thus $b^2\in H$. However $b\not\in H$, by \Cref{bnotH}, and hence $(ww^{b^2})^{-1} = (ww^{b^2})^b\in W\cap W^b=1$. 
Therefore $w^{b^2}=w^{-1}$ for every $w \in W$.
As the action of $H/K$ on $W$ is faithful, $Kb^2$ is a central element of order $2$ of $H/K$, in contradiction with \Cref{StructureHK}.
\end{proof}

 Now we put together all the information gathered to obtain a contradiction. By \Cref{centralizerb}, $G$ has an element of order $4\cdot 7$ while a Sylow $2$-subgroup of $G$ is cyclic by (\ref{S-s}). This is a contradiction with \Cref{CutOrders1}(\ref{noelements4p}).

\section{Graph (t)}\label{Section-t}

In this section we prove that the graph (t) in \Cref{QuestionE} is not the GK-graph of a finite solvable \cut group. The strategy is similar to the one of the previous section, so we fix a solvable \cut group $G$ of minimal order with 
\begin{equation}\label{Graph-t}
\GK(G)=\begin{tikzpicture}[local bounding box=bb,baseline=(bb.center)]
\node[label=west:{$2$}] at (0,0.5) (2){};
\node[label=east:{$3$}] at (0.5,0.5) (3){};
\node[label=west:{$5$}] at (0,0) (5){};
\node[label=east:{$7$}] at (0.5,0) (7){};
\foreach \p in {2,3,5,7}{
	\draw[fill=black] (\p) circle (0.075cm);
}
\draw (2)--(3);
\draw (2)--(5);
\draw (2)--(7);
\draw (3)--(5);
\end{tikzpicture} 
\end{equation} 

Arguing as in the previous case we get again $G=V\rtimes S$ for $V$ an elementary abelian Sylow $p$-subgroup of $G$ and $S$ a Hall $p'$-subgroup of $G$, with $p\in \{5,7\}$. Moreover, $F(G)=V$ and the action of $S$ on $V$ is faithful and irreducible. The symmetry  between the graphs (s) and (t) is reflected in the first three facts that we prove, but the proofs are completely different.
 
\setcounter{fact}{0}

\begin{fact}\label{Casotp=5}
	$p=5$.
\end{fact}

\begin{proof}
It suffices to show that $p\ne 7$ and again we argue by contradiction, so suppose $p=7$.
Let $\{q,r\}=\{3,5\}$. 
By \Cref{CutOrders1}\eqref{cyclicsylow}, $G_q$ and $G_r$ are cyclic of orders $q$ and $r$, respectively.

Let $X$ be the group generated by the elements of order $q$ of $S$. 
By \Cref{GSylowp}, $X=X'\rtimes C_q$. 
We claim that $X'=O_2(X)$. Otherwise $r\in \pi(X')$ and $X$ has a normal subgroup $N$ with $N=\GEN{x,O_2(X)} \subseteq X'$ and $r=|x|$. Then $N/O_2(X)$ is a normal cyclic subgroup of order $r$ in $X/O_2(X)$ and hence, $N/O_2(X)$ commutes with every element of order $q$ in $X/O_2(X)$. In particular, $N/O_2(X)$ is central in $X/O_2(X)$. Thus, there exists $O_2(X)\le L\leq X$ such that $r\nmid |L|$ and 
	$$X/O_2(X)=N/O_2(X) \times L/O_2(X).$$ 
Hence, every element of order $q$ of $X/O_2(X)$ belongs to $L/O_2(X)$, so $X=L$, a contradiction. 

Now let $a\in S$ of order $3\cdot 5$. Let $X$ (respectively, $Y$) be the group generated by the elements of order $3$ (respectively, $5$) of $S$. 
By the previous paragraph, $X'=O_2(X)$, $Y'=O_2(Y)$, $X=X'\rtimes \GEN{a_3}$ and $Y=Y'\rtimes \GEN{a_5}$. 
By the Frattini argument, $S=Y\N_S(\GEN{a_5})$ and $\N_S(\GEN{a_5})=X_0\N_{\N_S(\GEN{a_5})}(\GEN{a_3})$, where $X_0$ is the group generated by the elements of order $3$ of $\N_S(\GEN{a_5})$ and hence $X_0\subseteq X$. 
Then 
	\[S=Y\N_S(\GEN{a_5})=XY\N_{\N_S(\GEN{a_5})}(\GEN{a_3})=
	XY\N_S(\GEN{a})  =X'Y'\N_S(\GEN{a}).\]

\textbf{Claim 1}. $a_3$ is rational in $S$ and $\GEN{a_3}$ is not normal in $S$.

Assume that $a_3$ is not rational in $S$. Then, as  $a$ is inverse-semi-rational $\N_S(\GEN{a})/\C_S(a)$ is cyclic of order $4$ generated by an element $t\C_S(a)$ for some $t\in S$ with $a_3^t=a_3$ and $a_5^t=a_5^2$.  
Therefore, 
\[S=X'Y'\N_S(\GEN{a})=X'Y'(\GEN{a}\times \C_S(a)_2)\GEN{t},\quad 
O_2(S)=X'Y'\C_S(a_2) \qand S/O_2(S)\cong \GEN{a}\rtimes \GEN{t}.\]
Then $S/\GEN{a_5}O_2(S)\cong \C_{12}$ which is not \cut, a contradiction. 
So $a_3$ is rational in $S$.

Suppose that $\GEN{a_3}$ is normal in $S$. Then $a_3$ and $a_3^2$ are the unique $3$-elements of $S$. Since every element $v$ of $V$ is inverse semi-rational, either $v^{a_3}=v^2$ or $v^{a_3^2}=v^2$. In the second case $v^{a_3}=v^4$. If $v_1,v_2\in V\setminus\{1\}$ with $v_1^{a_3}=v_1^2$ and $v_2^{a_3}\neq v_2^2$, then $(v_1v_2)^{a_3} = v_1^2 v_2^4\not\in \{(v_1v_2)^2,(v_1v_2)^4\}$, a contradiction. Therefore, either $v^{a_3}=v^2$ for every $v\in V$, or $v^{a_3}=v^{4}$ for every $v\in V$. 
As $V$ is faithful, this means that $a_3$ is central in $S$ and hence it is not rational in $S$, in contradiction with the previous paragraph. This finishes the proof of Claim 1.

\medskip
\textbf{Claim 2}. If $V=W^S$ for some subgroup $H$ of $S$ and an $\F_7H$-module $W$, then $H=S$. 

To prove this we first prove that $H$ contains $\N_S(\GEN{a_3})\setminus \C_S(a_3)$. Indeed, let $t\in \N_S(\GEN{a_3})\setminus \C_S(a_3)$ and suppose that $t\not \in H$. Then $a_3^t=a_3^{-1}$ and there is a right transversal $T$ of $H$ in $S$ containing $1$ and $t$. Therefore  $V=W^S=\bigoplus_{x\in T}W^x$.
Fix $w\in W\setminus \{1\}$. 
As $ww^t$ is inverse semi-rational in $G$, there exists $s\in S$ such that $(ww^t)^s=(ww^t)^2$. 
We can assume that $s$ is a $3$-element of $S$ and hence that $s$ has order $3$. 
By \Cref{InducedFPF}, $s\in X\subseteq H$,  so $w^s=w^2$ and $w^{s^{t^{-1}}}=w^2$. 
Since $s\in X=X'\rtimes \GEN{a_3}$, we write $s=x\delta$ for some $x\in\GEN{a_3}\setminus \{1\}$ and $\delta \in X'\trianglelefteq S$. 
Thus, $$w^2=w^s=w^{x\delta} \qand w^2=w^{s^{t^{-1}}}=w^{x^{t^{-1}}\delta^{t^{-1}}}=w^{{x^2}\delta^{t^{-1}}}.$$ 
We define $\delta'=(\delta^x \delta)^{-1}\delta^{t^{-1}}\in X'$ so that 	
$$w^2=w^{x^2\delta^x \delta \delta'}=w^{(x\delta)^2\delta'}=w^{s^2\delta'}=(w^4)^{\delta'}.$$ 
It follows that $w^{\delta'}=w^4$ and hence $3\mid |\delta'|$, a contradiction, since $\delta'$ is a $2$-element. 
So, indeed $\N_S(\GEN{a_3})\setminus \C_S(\GEN{a_3})\subseteq H$. 

Now let $t\in \N_S(\GEN{a_3})\setminus \C_S(\GEN{a_3})$, which exists by Claim 1, and let $x\in \C_S(a_3)$. 
Then $xt\in \N_S(\GEN{a_3})\setminus \C_S(\GEN{a_3})\subseteq H$ so $x\in H$. This proves that  
$\N_S(\GEN{a_3})=\C_S(a_3)\GEN{t}\subseteq H$. 
Moreover, $X\subseteq H$, by \Cref{InducedFPF} and, from the Frattini argument, we get $S=X\N_S(\GEN{a_3})=H$, as desired. This finishes the proof of the Claim 2.

\medskip
As the action of $S$ on $V$ is faithful, Claim 2 and \Cref{HAbelCic}, imply that $V_N$ is homogeneous for every normal subgroup $N$ of $S$ and if $N$ is abelian, then it is cyclic. Hence the Fitting subgroup of $S$ is as in \Cref{StructureFED}. We adopt the notation of that proposition, i.e. $F=F(S)$, $F=ED$, $Z=E\cap D$, $U\le D$, etc. 
Since, by Claim~1, $\GEN{a_3}$ is not normal in $S$, $X'$ is a non-trivial normal $2$-subgroup of $S$ and hence the order of $F$ is even.
Then $Z_2$ is cyclic of order $2$ and either $E_2=Z_2$ or $E_2$ is extra-special of exponent $2$ or $4$.

\medskip
\textbf{Claim 3}. Every element of order $3$ of $S$ centralizes $D$. 

Let $x$ be an element of order $3$ of $S$ and $y\in D$. 
Since $U$ is normal in $S$ and cyclic of order $2^k$ or $2^k\cdot 5$ for some $k\in\mathbb{N}$, $x$ centralizes $U$. Moreover, $[D:U]\le 2$ and hence there exists $u\in U$ such that $y^x=yu$. Then $y=y^{x^3}=yu^3$, so $u^3=1$ and hence $u=1$ and the claim follows.
\medskip

We show that $X'\subseteq E_2$ and $X\subseteq E_2\rtimes \GEN{a_3}$. Indeed, $X'\subseteq F$, as $X'$ is a  normal $2$-subgroup of $S$. Thus $X=X'\rtimes \GEN{a_3}\subseteq F\rtimes \GEN{a_3}=ED\rtimes \GEN{a_3}$. Since $D$ commutes with $\GEN{a_3}$ and $E$ we have that $E\rtimes \GEN{a_3}$ contains all the elements of order $3$ of $F\rtimes \GEN{a_3}$, i.e., $E\rtimes \GEN{a_3}$ contains all the elements of order $3$ of $S$. Then $X\subseteq E\rtimes \GEN{a_3}$, so $X'\subseteq E_2$ and $X\subseteq E_2\rtimes \GEN{a_3}$.

Let $x\in S$ be an arbitrary element of order $3$. By Claim 1 and since $\GEN{x}$ is conjugate to $\GEN{a_3}$ in $S$, $x$ is rational in $S$. In particular, there is $t\in S$ with $x^t=x^{-1}$, and hence, the map $v\mapsto v^{t^{-1}}$ is a bijection from $V_{x}(2)$ to $V_{x}(4)=V_{x^{-1}}(2)$.  Thus, $2\cdot dim_{\F_7}V_{x}(2)\leq dim_{\F_7}V$ and hence $|V_{x}(2)|^2\leq |V|.$ 

On the other hand, as every $v\in V$ is inverse semi-rational in $S$, for every $v\in V$, there is $x\in S$ of order $3$ such that $v^x=v^2$, i.e. 
\[V\subseteq \bigcup_{x\in S, \ |x|=3} V_{x}(2).\]
So, if $k$ is the number of elements of order $3$ in $S$, then 
$|V|-1\leq k(|V|^{1/2}-1)$, so that $|V|^{1/2}<k$. 
Moreover, as $X\subseteq E_2\rtimes \GEN{a_3}$, every element of order $3$ of $S$ is of the form $ea_3$ or $ea_3^2$ with $e\in E_2$. Hence, $|V|^{1/2}< k\le 2\cdot|E_2|$. We conclude that $|V|<4\cdot |E_2|^2$.

Consider $V$ as an $\F_7S$-module and observe that $E_2$ is a normal subgroup of $S$, since $E$ is normal in $S$ and $E_2$ is characteristic in $E$.    As
  mentioned in the paragraph prior to Claim 3, this implies that $V_{E_2}$ is homogeneous, i.e. $V_{E_2}\cong W^n$ for an irreducible $\F_7E_2$-module $W$ and some positive integer $n$.

As the exponent of $E_2$ is $2$ or $4$ and $\F_{7^2}$ has a fourth root of unity, $K:=\F_{7^2}$   is a splitting field for $E_2$. 
Therefore, $W\otimes_{\F_7}K$ has an absolutely irreducible $KE_2$-module $M$ and the constituents of $W\otimes_{\F_7}K$ are the Galois conjugates of $M$ over $\F_7$. Let $m$ denote the number of these Galois conjugates. Then $m\le 2$ and as $V_S$ is faithful, so are $W_{E_2}$ and $M_{E_2}$. 
The degree of an irreducible and faithful ordinary character of $E_2$ is precisely $e$ where $e^2=|E_2:Z_2|$, so $dim_K M=e$. 
Therefore, 
$$\dim_{\F_7} V = n \dim_{\F_7} W = n \dim_K W\otimes_{\F_7}K = nm\dim_K(M)=nme.$$ 
Now since $|E_2|=2e^2$ we have that
	\begin{equation}\label{e} 7^{nme}=|V|<4 \cdot |E_2|^2 = 16 e^4. \end{equation}
As $7^{2\ell}\ge 16\ell^4$ for every positive integer $\ell$, necessarily $n=m=1$ and hence $e=\dim_{\F_7} V$ and $7^e<16e^4$. The latter implies that $e\le 4$ and, as $e$ is a power of $2$, that $dim_{\mathbb{F}_7}V\in \{1,2,4\}$.

However, as $S$ has a cyclic subgroup of order $5$ acting without fixed points on $V\setminus \{1\}$, necessarily $4$ divides $dim_{\mathbb{F}_7}V$ and thus, $dim_{\mathbb{F}_7}V=4$. Then $S\leq \GL(4,7)$. Note that all subgroups of order $5$ of $\GL(4,7)$ are conjugate and if $C$ is one of them, then the unique Sylow $3$-subgroup of $\N_{\GL(4,7)}(C)$ is central in $\GL(4,7)$. Therefore, $a_3\in Z(S)$, in contradiction with Claim 1.
\end{proof}

So $V$ is an elementary abelian normal Sylow $5$-subgroup of $G$ and $S$ is a Hall $5'$-subgroup of $G$ acting faithfully and irreducibly on $V$. Clearly $\GK(S)= (3-2-7)$. Moreover, by \Cref{CutOrders1}\eqref{cyclicsylow}, the Sylow $7$-subgroups of $S$ have order $7$.

\begin{fact}\label{C7notnormal}
	The Sylow $7$-subgroups of $S$ are not normal in $S$.
\end{fact}

\begin{proof} 
In this proof we use additive notation for $V$, so that we consider $V$ as a (right) $\F_5S$-module. As, by \Cref{CutOrders1}\eqref{5rational}, every element of $V$ is rational in $G$, $V_S$ has the eigenvector property. 

We fix a Sylow $7$-subgroup $\GEN{a}$ of $S$.
The action of $\GEN{a}$ on $V$ is fixed-point-free, in particular, every non-zero submodule of $V_{\F_5\GEN{a}}$ is faithful, so that its simple submodules have dimension $6$. Using this it is easy to see that if $v\in V\setminus \{1\}$, then $v,va,va^2,\dots,va^5$ is a basis of $v\F_5\GEN{a}$.
Indeed, $v\F_5\GEN{a}$ is generated by $v,v^a,\dots,v^{a^6}$ and $w=v+va+\dots+va^6$ commutes with $a$. Thus 
	\[v+va+\dots+va^6=w=0,\]
as $G$ has not elements of order $5\cdot 7$.  In particular, $\{va^{i-1}\}$ for $i=1,\dots,6$ is a basis of $v\F_5\GEN{a}$.

Suppose that $\GEN{a}$ is normal in $S$. Then $S$ satisfies the hypothesis of \Cref{rationalS2}. So $S=\GEN{a}\rtimes (\GEN{b}\times Q)$ with $|b|=3$, $a^b=a^2$ and $Q$ a rational $2$-group. Let $C=\C_Q(a)$. Clearly $[Q:C]\le 2$ and $a^g=a^{-1}$ for every $g\in Q\setminus C$. 
Observe that the set of $2$-elements of $S$ is the disjoint union of  $C$ and  $Q_1=\{xq : q\in Q\setminus C, x\in \GEN{a}\}$. Moreover, if $g$ is a $2$-element of $S$, then 
$$a^g = \begin{cases} a, & \text{if } g \in C; \\a^{-1}, & \text{if } g\in Q_1.\end{cases}$$ 
Let $g\in C\cup Q_1$ and $v\in V_g(2)\setminus \{0\}$.
Let $W=v\mathbb{F}_5\GEN{a}$ and write $w_i=va^{i-1}=v^{a^{i-1}}$ for $i=0,...,6$. By the second paragraph of the proof, $B=\{w_i\}_{i\geq 1}^6$ is a basis of the vector space $W$. 
Let 
$$A=
\begin{pmatrix}
	2&0&0&0&0&0\\
	-2&-2&-2&-2&-2&-2\\
	0&0&0&0&0&2\\
	0&0&0&0&2&0\\
	0&0&0&2&0&0\\
	0&0&2&0&0&0 \end{pmatrix}.$$ 
We claim that the row matrix representation $\rho_v(g)$ of the action of $g$ on $W$ in the basis $B$ is as follows:
\[\rho_v(g)=\begin{cases} 2I, & \text{if } g\in C; \\ A, & \text{if } g\in Q_1.\end{cases}\]
This is straightforward if $g\in C$ and a consequence of the following calculation if $g\in Q_1$:
$$\begin{aligned}w_1g&=2w_1,\\
	w_0g&=w_1a^{-1}g=w_1ga=2w_1a=2w_2,\\
	w_6g&=w_0a^{-1}g=w_0ga=2w_2a=2w_3,\\
	w_5g&=w_6a^{-1}g=w_6ga=2w_3a=2w_4,\\
	w_4g&=w_5a^{-1}g=w_5ga=2w_4a=2w_5,\\
	w_3g&=w_4a^{-1}g=w_4ga=2w_5a=2w_6,\\
	w_2g&=w_3a^{-1}g=w_3ga=2w_6a=2w_0=-2(w_1+\dots+w_6).\\
\end{aligned}$$

\medskip
\textbf{Claim}: $V_C$ has the eigenvector property.


Let  $v\in V\setminus \{0\}$ and let $W$ and $B$ as above. As $V_S$ has the eigenvector property, there is $g\in C\cup Q_1$ such that $v\in V_g(2)$. If $g\in C$, then we are done. So suppose that $g\in Q_1$. Then $\rho_v(g)=A$. 

Now let $u=w_0+2w_1\in W$ and write $u_i=ua^{i-1}=u^{a^{i-1}}$ for $i=0,...,6$. As before $B'=\{u_i\}_{i\geq 1}^6$ is another basis of $W$. 
Again by the eigenvector property of $V_S$, there exists $h\in C\cup Q_1$ such that $uh=2u$. 
If $h\in C$, then $\rho_u(h)=2I$ and hence the action of $h$ on $W$ is given by multiplication by $2$, in particular $vh=2v$, and we are done again. So, suppose that $h\in Q_1$. Then $\rho_u(h)=A$. 
Let 	
$$U=\begin{pmatrix}
	1&-1&-1&-1&-1&-1\\
	1&2&0&0&0&0\\
	0&1&2&0&0&0\\
	0&0&1&2&0&0\\
	0&0&0&1&2&0\\
	0&0&0&0&1&2 \end{pmatrix}.$$
A straightforward computation shows that the rows of $U$ are the coefficients of the elements of $B'$ expressed in the basis $B$. Hence, by elementary linear algebra, conjugating by $U$ transforms the matrix representing the action of $g$ in the basis $B$ into the matrix $A'$ representing this action in the basis $B'$, i.e. $A'=U^{-1}AU$. 
Thus $\GEN{A,A'}\cong \GEN{h,g}$ but  $|\GEN{A,A'}|=2^2\cdot 3^2\cdot 7$ and hence $|S|$ is multiple of $9$, a contradiction. 
This finishes the proof of the Claim.
\medskip

Therefore $V_Q$ also has the eigenvector property and hence $V\rtimes Q$ satisfies the hypothesis of \Cref{VS2} (with $Q$ playing the role of the $2$-group $G$). 
So $dim_{\mathbb{F}_5}V=2n$, for some positive integer $n$, $Q\cong Q_8 \wr K$ for some $2$-subgroup $K$ of $\Sigma_n$ and $V\rtimes Q\cong \MM\wr K$.

We can choose a basis $\{v_1,v_2,v_3,v_4,\dots,v_{2n-1},v_{2n}\}$ of $V$ such that the representation of some $g\in Q\cong Q_8^n\rtimes K$ in this basis has the form 
$$\begin{pmatrix}2&0\\0&3\end{pmatrix}\oplus I_{2n-2}.$$ 
Note that $V_g(2)$ has dimension $1$. 

Put $W=v_1\mathbb{F}_5\GEN{a}$ and $w_i=v_1a^{i-1}$ for $i=0,\dots,6$. Recall that $B=\{w_i\}_{i=1}^6$ is a basis of $W$ and the matrix representation of the restriction of $g$ to $W$ in this basis is $A$ if
$a^g=a^{-1}$ and $2I_6$ if $g\in C$. In either case, this is a contradiction with the previous paragraph because the eigenspace of eigenvalue $2$ has dimension 3 for $A$, and dimension 6 for $2I_6$.
\end{proof}

Let $X$ be the subgroup of $S$ generated by the elements of order $7$ in $S$. By \Cref{GSylowp}, $X=X'\rtimes \GEN{a}$ for an element $a$ of order $7$ and, by \Cref{C7notnormal}, $X'$ is a non-trivial normal $\{2,3\}$-subgroup of $S$. 

\begin{fact}\label{facttX27}
$X'$ is a $2$-group and $F(X)=X'\ne 1$.
\end{fact}

\begin{proof}
It suffices to show that $|X'|$ is not divisible by $3$. 
Suppose by contradiction that $3$ divides the order of $X'$. 
Let $\mathcal S_3$ be the set of Sylow $3$-subgroups of $X'$. Then $|\mathcal S_3|$ is a power of $2$. Moreover $\GEN{a}$ acts by conjugation on $\mathcal S_3$, and therefore it fixes at least one Sylow $3$-subgroup, say $L$.  Then $\GEN{a}$ acts by conjugation  without fixed points on $L$ and on $L_0=\Omega(\Z(L))$. 
So we can see the latter as an $\F_3 \GEN{a}$-module, which is actually homogeneous, because $\F_3\GEN{a}$ has a unique simple module, up to isomorphism, on which $\GEN{a}$ acts fixed-point-freely. Moreover, this simple module has dimension $6$. 
Therefore, if $e$ is a non-trivial element of one of the simple submodules $T$ of $L_0$, then $T=\GEN{e,a}\cong C_3^6\rtimes_{\text{Fr}} C_7$. 
As the action of $S$ on $V$ is faithful, so is the action of $T$ on $V$. Then \Cref{FrobeniusFaithful} yields that $a$ does not act fixed-point-freely on $V$, in contradiction with $G$ having no elements of order $5\cdot 7$.	
\end{proof} 

Therefore, $X'$ is a non-trivial normal $2$-subgroup of $S$ such that $X/X'$ is a normal Sylow subgroup of $S/X'$ of order $7$ and $\GK(S/X')\subseteq (3-2-7)$. Then, by \Cref{rationalS2}, 
\begin{equation}\label{SX'}
	S/X'\cong (C_7\rtimes (C_3\times Q)),	
\end{equation}
with $Q$ a $2$ group. Thus $F(S)$ is a $2$-group.

Let $H$ be a minimal subgroup of $S$ such that, as $S$-module, $V$ is induced from an $H$-module $W$, and let $K$ be the kernel of the action of $H$ on $W$.  
We use the standard bar notation for the projection of $H$ onto $H/K$ and $\widetilde \cdot$ for the projection of $\overline H$ onto $\overline H/\Phi(F(\overline H))$. To avoid cumbersome notation $\widetilde{\overline h}$, we abbreviate this as $\widetilde h$.

By \Cref{HAbelCic}, every abelian normal subgroup of $\overline{H}$ is cyclic and so, $\overline{H}$ contains normal subgroups $F, D,E,U$ and $Z$ satisfying the conditions of   \Cref{StructureFED}.
Note that $X\subseteq H$ by \Cref{InducedFPF}. As $V_S$ is faithful, $\Core_S(K)=1$ so since $X'$ is normal in $S$, $\overline{X}'\ne 1$ and therefore, $|F|$ is not multiple of $7$. Moreover, as $G$ does not contain elements of order $3\cdot 7$, it follows that $|Z|$ is not multiple of $3$ and the same follows for $|F|$. So $F$ is a $2$-group.
Observe that $Z=\Omega(Z(F))$ has order $2$. Hence $Z$ is contained in every non-trivial normal subgroup of $F$. 
As $Z\subseteq U$, $F=ED$, $(E,D)=1$ and $E/Z$ and $D/U$ are elementary abelian, it follows that $F/U$ is elementary abelian, hence $\Phi(F)\subseteq U$, and so $\Phi(F)$  is cyclic. 
As $\overline X$ is not abelian and $X'$ is a $2$ group, $\overline X'$ is not cyclic and hence it is not abelian. Thus $\overline X''\ne 1$ and hence $Z\subseteq \overline X''\subseteq \Phi(F)$. 
	
So $\overline X'$ is not abelian, while $\Phi(F)$ is cyclic and so $\widetilde X'\ne 1$. 
Consider $\widetilde F$ as an $\F_2 \widetilde{H}$-module. 
Now,    $\widetilde X$ is once again a group generated by elements of order $7$ satisfying the hypothesis of  \Cref{GSylowp},  so $\widetilde{a}$ acts by conjugation without fixed points on   $\widetilde X'$, and hence every irreducible constituent of $\widetilde X'$ as $\F_2\GEN{\widetilde a}$-module has dimension $3$ and thus $\dim \widetilde X' \ge 3$. 
On the other hand $F=ED$ and $E\cap D=Z\subseteq \Phi(F)$ and therefore 
\[\widetilde{F} = \widetilde E \oplus \widetilde D.\]
Observe that  $\dim \widetilde D \le 2$, as $D$ is generated by at most $2$ elements.  
Thus 
	\begin{eqnarray*}
	\dim(\widetilde X'\cap \widetilde E) &=& \dim(\widetilde X'\cap \widetilde E) + \dim(\widetilde F) - (\dim(\widetilde E)+\dim(\widetilde D))  \\
	&\ge & \dim(\widetilde X'\cap \widetilde E) + \dim(\widetilde X'+\widetilde E) - (\dim(\widetilde E)+\dim(\widetilde D)) = \dim (\widetilde X')-\dim (\widetilde D) \ge 1.
	\end{eqnarray*}
Thus there is an element $x\in X'$ such that $\overline x=eu$ for some $e\in E\setminus \Phi(F)$ and some $u\in \Phi(F)$.  
Again, by \Cref{GSylowp}, $a\overline X''$ acts without fixed points on $\overline X'/\overline X''$ and therefore $(\overline{ax})^7\in \overline X''$. Moreover, as $\Phi(F)\subseteq U$ and $U$ is a normal cyclic $2$-subgroup of $\overline{H}$, we deduce that $\overline a$ commutes with $u$. Thus 
\[\overline x^{a^6} \overline x^{a^5} \dots   \overline x^a \overline x =  e^{a^6} e^{a^5}\dots e^a  e u^7=(\overline a \overline x)^7\in \overline X''.\]
Since $\overline X''\subseteq \Phi(F)\subseteq U$, we conclude that $(\overline ae)^7=e^{a^6}\dots e^a e \in U\cap E \subseteq D\cap E=Z$. 
So $|\overline a e|\in \{7,14\}$. If $|\overline a e|=7$, then $e\in \overline X'$, by \Cref{GSylowp}. Suppose that $|\overline a e|=14$. Then $(\overline a e)^2 = \overline a^2 e^{a}e $ has order $7$, so   $e^{\overline a} e\in \overline X' \cap E$. 
We claim that $e^{\overline a}e\not\in\Phi(F)$. 
Suppose otherwise. 
Then $e^{\overline a}e\in E\cap \Phi(F)=Z$. As $e^2\in E\cap \Phi(F)=Z$, it follows that $e^{\overline a}\in eZ$, so that  $e^{\overline{a}}=e$ or $e^{\overline{a}}=ez$ for $Z=\GEN{z}$. 
In the latter case, $e=ez^2=e^{\overline a^2}$ and as $\overline a$ has order $7$, we deduce that actually the first case holds. But then $e^7=(\overline a e)^7 \in Z\subseteq \Phi(F)$, against the election of $e$. 
So, in both cases  $\overline X'\cap E\not\subseteq \Phi(F)$ and in particular $\overline X'\cap E\not\subseteq \overline X''$.

Then $M=(\overline X'\cap E)\overline X''/\overline X''$ is a nontrivial normal subgroup of $\overline X/\overline X''=(\overline X'/\overline X'')\rtimes \GEN{\overline a \overline X''}$. Once more the latter satisfies the conditions of \Cref{GSylowp}, so that $a$ acts without fixed points on $\overline X'/\overline X''$ and hence so does on $M$. 
As $Z\subseteq \overline{X}''$ and $E/Z$ has exponent $2$, also $M$ has exponent $2$.
Thus we can see $M$ as an $\F_2C_7$-module and we fix a simple submodule $N$ of $M$. Then $\GEN{N,\overline a \overline X''}=\GEN{e \overline X'',\overline a\overline X''}$ for some $e\in \overline X'\cap E$, and $\GEN{e \overline X'',\overline a\overline X''}$ is a Frobenius group of the form $C_2^3\rtimes_{\text{Fr}} C_7$, contained in $(\overline X'\cap E)\GEN{\overline a}\overline X''/\overline X''$. The latter is naturally isomorphic to $(\overline X'\cap E)\GEN{\overline a}/(\overline X'\cap E)\GEN{\overline a})\cap \overline X''$.
	Moreover,  $Z\subseteq E\cap \overline X''  \subseteq E\cap \Phi(F) =Z$ and hence $(\overline X'\cap E)\GEN{\overline a})\cap \overline X''=E\cap \overline X''=Z$. Thus $\GEN{eZ,\overline aZ}\cong \GEN{e \overline X'',\overline a\overline X''}\cong C_2^3\rtimes_{\text{Fr}} C_7$.

As $|Z|=2$, $\GEN{e,\overline a, Z}\cong R\rtimes C_7$, where $R$ is  a group of order $2^4$ fitting a short exact sequence
	$$1\to C_2\to R\to C_2^3\to 1. $$   
If the sequence does not split, then $R$ is one of the following groups: $D_8\times C_2$, $Q_8\times C_2$, $(C_4\times C_2) \rtimes C_2$, $C_2^2\times C_4$. None of these groups has an automorphism of order $7$, hence $\GEN{e,\overline a,Z}\cong R\times C_7$, in contradiction with 
$\GEN{e,\overline a, Z} /Z\cong \GEN{eZ,\overline aZ}\cong C_2^3\rtimes _{\text{Fr}} C_7$. 
Thus the sequence splits, hence $R\cong C_2^4$, and, as $\overline a$ and the elements of $Z$ commute,   $\GEN{e,\overline a}\cong C_2^3\rtimes_{\text{Fr}} C_7$. Then we look at $W$ as an $\F_5\GEN{e, \overline a}$-module. 
As the action of $\overline H$ on $W$ is faithful, so is the action of $\GEN{e,\overline a}$. Then, by \Cref{FrobeniusFaithful}, $\GEN{\overline{a}}$ does not act fixed-point-freely on $W$, in contradiction with $G$ not having elements of order $5\cdot7$.

\medskip

\textbf{Acknowledgment}: \Cref{FabelianSylow}, \Cref{G2C2ornoelements21}  and \Cref{F7elementaryabelian} in \Cref{Section-s} were discovered during the preparation of \cite{BKMdR}. We thank Ann Kiefer, Andreas Bächle and Sugandha Maheshwary for allowing us to use them here.

\bibliographystyle{amsplain}
\bibliography{References}

	\end{document}